\documentclass[11pt, final,reqno,a4paper]{amsart}
\usepackage{amsmath,amsthm,amssymb,amsfonts,mathrsfs,hhline,cite,enumerate,mathtools}

\usepackage{caption,subcaption,graphicx}
\usepackage[subrefformat=parens, labelfont=up]{subcaption}
\usepackage{xspace}
\usepackage[dvipsnames]{xcolor}

\usepackage[colorlinks]{hyperref}

\usepackage{verbatim}
\usepackage{enumitem}
\usepackage{tikz}
\setlist[enumerate,1]{label=(\arabic*)}
\usepackage{multicol}
\usepackage{geometry}
\numberwithin{equation}{section}
\theoremstyle{plain} 
\newtheorem{thm}{Theorem}[section]
\newtheorem{lem}[thm]{Lemma}
\newtheorem{cor}[thm]{Corollary}
\newtheorem{prop}[thm]{Proposition}
\theoremstyle{definition}
\newtheorem{defn}[thm]{Definition}
\newtheorem{ques}[thm]{Question}
\newtheorem{rem}[thm]{Remark}
\newtheorem{ex}[thm]{Example}

\newcommand\nc\newcommand
\nc\rnc\renewcommand

\nc\R{\mathcal R} 
\nc\T{\mathcal T}
\nc\V{\mathcal V}
\nc\U{\mathcal U}
\nc\G{\mathcal G}
\rnc\P{\mathcal P}
\nc\B{\mathcal B}
\rnc\S{\mathcal S}
\nc\TL{\mathcal T\!\mathcal L}
\nc\K{\mathcal K}
\nc\C{\mathcal C}
\nc\CC{\mathscr C}
\nc\D{\mathcal D}
\nc\I{\mathcal I}
\nc\M{\mathcal M}
\nc\bd{{\bf d}}
\nc\br{{\bf r}}
\nc\bit{\begin{itemize}}
\nc\eit{\end{itemize}}
\nc\ben{\begin{enumerate}[label=\textup{(\roman*)},leftmargin=10mm]}
\nc\bena{\begin{enumerate}[label=\textup{(\alph*)},leftmargin=10mm]}
\nc\een{\end{enumerate}}
\nc\set[2]{\{#1:#2\}}
\nc\bigset[2]{\big\{#1:#2\big\}}
\nc\io\iota
\nc\ve\varepsilon
\nc\NN{\mathbb N}
\nc\ZZ{\mathbb Z}
\nc\PP{\mathbb P}
\nc\BB{\mathbb B}
\nc\CCC{\mathbb C}
\nc\op\oplus
\nc\ot\otimes
\rnc\implies{\ \Rightarrow\ }
\rnc\iff{\ \Leftrightarrow\ }
\nc\IFF{\qquad \Leftrightarrow\qquad }
\nc\even{{\operatorname{even}}}
\nc\odd{{\operatorname{odd}}}
\nc\al\alpha
\nc\la\langle
\nc\ra\rangle
\nc\sm\setminus
\nc\lam\lambda
\nc\Om\Omega
\nc\de\delta
\nc\pr\star

\nc\bz\varnothing

\nc\id{\operatorname{id}}
\nc\bm{[m]}
\nc\bn{[n]}
\nc\AND{\qquad\text{and}\qquad}
\nc\ANd{\quad\text{and}\quad}
\nc\COMMA{,\qquad}
\nc\COMMa{,\quad}
\nc\ol\overline
\nc\ool[1]{\ol{\ol{#1}}}
\nc\es\emptyset
\nc\sub\subseteq
\nc\mt\mapsto
\nc\dom{\operatorname{dom}}
\nc\codom{\operatorname{codom}}
\nc\coker{\operatorname{coker}}
\nc\rank{\operatorname{rank}}
\nc\X{\mathcal X}
\nc\Y{\mathcal Y}
\nc\Z{\mathcal Z}

\nc\darcx[3]{\draw(#1,0)arc(180:90:#3) (#1+#3,#3)--(#2-#3,#3) (#2-#3,#3) arc(90:0:#3);}
\nc\uarcx[3]{\draw(#1,2)arc(180:270:#3) (#1+#3,2-#3)--(#2-#3,2-#3) (#2-#3,2-#3) arc(270:360:#3);}
\nc\darc[2]{\darcx{#1}{#2}{.4}}
\nc\uarc[2]{\uarcx{#1}{#2}{.4}}
\nc\udotted[2]{\draw[dotted] (#1+.5,2)--(#2-.5,2);}
\nc\ddotted[2]{\draw[dotted] (#1+.5,0)--(#2-.5,0);}
\nc\vertlab[2]{\node () at (#1,2.7) {\tiny $#2$};}
\nc\lvertlab[2]{\node () at (#1,-.7) {\tiny $#2$};}

\nc\fourrow[4]{#1&#2&#3&#4}

\nc\lv[1]{\fill (#1,0)circle(.15);}
\nc\lvs[1]{{\foreach \x in {#1}{\lv{\x}}}}
\nc\uv[1]{\fill (#1,2)circle(.15);}
\nc\uvs[1]{{\foreach \x in {#1}{\uv{\x}}}}
\nc\stline[2]{\draw(#1,2)--(#2,0);}
\nc\stlines[1]{{\foreach \x/\y in {#1}{\stline{\x}{\y} }}}

\nc\lvc[2]{\fill[#2] (#1,0)circle(.15);}
\nc\uvc[2]{\fill[#2] (#1,2)circle(.15);}
\nc\darcxc[4]{\draw[#4](#1,0)arc(180:90:#3) (#1+#3,#3)--(#2-#3,#3) (#2-#3,#3) arc(90:0:#3);}
\nc\uarcxc[4]{\draw[#4](#1,2)arc(180:270:#3) (#1+#3,2-#3)--(#2-#3,2-#3) (#2-#3,2-#3) arc(270:360:#3);}
\nc\darcc[3]{\darcxc{#1}{#2}{.4}{#3}}
\nc\uarcc[3]{\uarcxc{#1}{#2}{.4}{#3}}
\nc\stlinec[3]{\draw[#3](#1,2)--(#2,0);}

\nc\uvert[1]{\fill (#1,2)circle(.2);}
\rnc\lvert[1]{\fill (#1,0)circle(.2);}
\nc\custpartn[3]{{\lower1.4 ex\hbox{
\begin{tikzpicture}[scale=.3]
\foreach \x in {#1}
{ \uvert{\x}  }
\foreach \x in {#2}
{ \lvert{\x}  }
#3 \end{tikzpicture}
}}}
\nc\smallcustpartn[3]{{\lower1.2 ex\hbox{
\begin{tikzpicture}[scale=.2]
\foreach \x in {#1}
{ \uvert{\x}  }
\foreach \x in {#2}
{ \lvert{\x}  }
#3 \end{tikzpicture}
}}}
\nc\bpt[6]{\custpartn{1,2,3}{1,2,3}{\stline{#1}{#2}\uarc{#3}{#4}\darc{#5}{#6}}}
\nc\bpp[6]{\custpartn{1,2,3}{1,2,3}{\stline{#1}{#2}\stline{#3}{#4}\stline{#5}{#6}}}

\usepackage{calc}
\newcounter{ncols}
\newcounter{incols}
\newenvironment{partn}[1]{
  \setcounter{ncols}{#1} \setcounter{incols}{\thencols - 1}\setlength{\arraycolsep}{1pt}
  \Bigl( \hspace{-1.5truemm}\scriptsize 
    \begin{array}{@{\hskip 3pt} c *{\theincols}{|c} @{\hskip 3pt}  }
}{
     \end{array}
     \normalsize \hspace{-1.5truemm}\Bigr)\setlength{\arraycolsep}{6pt}
}

\begin{document}

~\vspace{-0.5cm} 

\title[Faithful representations of diagram categories]{Faithful linear and relational representations of diagram categories and monoids}
\date{\today}

\author[J. East]{James East}
\address{Western Sydney University}
\email{J.East@WesternSydney.edu.au}
\author[M. Johnson]{Marianne Johnson}
\address{University of Manchester}
\email{Marianne.Johnson@Manchester.ac.uk}
\author[M. Kambites]{Mark Kambites}
\address{University of Manchester}
\email{Mark.Kambites@Manchester.ac.uk}

\begin{abstract}
We study representations of diagram categories by binary relations and matrices over rings and semirings.  Our main result is a faithful involutive tensor representation of the partition category $\P$ (and consequently of each partition monoid~$\P_n$) by zero-one matrices over an arbitrary (additively) idempotent semiring. The dimensions of the matrices involved are powers of~$2$, and we show that these are minimal with respect to faithful involutive tensor representations by matrices over any semiring.  Intriguingly, these matrices encode the number of floating components formed when composing partitions, and can therefore be used to construct faithful representations of ($d$-)twisted partition categories $\P^\Phi$ and $\P^{\Phi,d}$ (and the respective twisted partition monoids $\P_n^\Phi$ and $\P_n^{\Phi, d}$) over rings of appropriate characteristic.  We also give lower-dimensional involutive representations of the Brauer and Temperley--Lieb categories~$\B$ and $\TL$.  In the case of $\TL$, the dimensions are given by Fibonacci numbers.


\end{abstract}
\maketitle

\vspace{-5mm}

\tableofcontents

\section{Introduction}\label{sect:intro}

Diagram algebras arise naturally in many areas of mathematics and science, and play important roles in statistical mechanics, biology, computer science, representation theory, logic, combinatorics, knot theory and more.  These algebras have bases consisting of various kinds of set partitions, which are represented and multiplied diagrammatically. Key examples include Brauer algebras \cite{B1937}, Temperley--Lieb algebras~\cite{TL1971} and partition algebras \cite{Martin1994,Jones1994}; for more on the background and applications, see the survey \cite{Martin2008}.  Early on, Jones and Kauffman recognised the importance of underlying diagram \emph{monoids} \cite{Jones1983,Jones1987,Jones1994,J1994, Kauffman1987,Kauffman1990, Kauffman1997}, which underpin the structure of the algebras, and the link was made precise by Wilcox in his work on twisted semigroup algebras \cite{W2007}; see also \cite{HR2005}.  Diagram monoids and algebras occur as endomorphism monoids/algebras of associated diagram \emph{categories}, studied for example in \cite{A2008,DDE2021,KST2024,Martin2008,Martin1994,ER2023,LZ2015}.

Diagram monoids have also become a major theme in modern semigroup theory, due to their intrinsically intriguing algebraic and combinatorial properties, and striking parallels with transformation semigroups.  For example, there are studies on Green's relations~\cite{W2007, FL2011}, presentations by generators and relations \cite{KM2006,E2018,E2021,MM2007,E2011}, characterisation and enumeration of idempotents \cite{DEEFHHKL2015, DEEFHHLM2019}, idempotent-generation \cite{DE2017,EF2012,EGPR2025,EG2017}, congruences \cite{EMRT2018, ER2022a, ER2022b, ER2022c}, Krohn--Rhodes complexity \cite{A2012}, semigroup identities \cite{A2014, ACHLV2015, KV2020, KV2024}, and connections to category theory \cite{EG2021,EPA2024}. 

In a different direction, the recent article \cite{CEM2024} studied transformation representations of diagram monoids, in each case giving an explicit formula for the \emph{transformation degree}, i.e.~the minimum number of points on which the monoid acts faithfully by functions.  For example, the partition monoid $\P_n$ has degree $1 + \frac{B(n+2)-B(n-1)+B(n)}2$, where $B(n)$ is the $n$th Bell number.  
Any transformation representation leads to a \emph{linear} representation of the same degree, over any unital (semi)ring.  In general, however, there can be faithful linear representations whose dimension is substantially lower than the transformation degree, and this is one of the main motivations for the present work.  As a sample result, we show that there is a faithful $2^n$-dimensional representation of $\P_n$ by matrices over any (additively) idempotent semiring.  This representation also models the well-known (product-reversing) involution of~$\P_n$ as transposition of matrices, thereby retaining information that is lost by transformation representations (as full transformation monoids have no involutions).  
Since matrices over the \emph{Boolean} semiring can be interpreted as binary relations, we also obtain a faithful representation of $\P_n$ by relations of degree~$2^n$.
In fact, all of these monoid representations can be packaged into a single faithful representation of the partition \emph{category} $\P$ in the full matrix category $\M$.  This category representation contains much more information, as it is defined also on hom-sets $\P_{m,n}$ with $m\not=n$.  It is additionally compatible with natural \emph{tensor} structures on $\P$ and~$\M$, and has order-theoretic connections as well.
We remark that the relation-theoretic interpretation of our representation $\P\to\M$ seems to be closely related to a special case of a functor considered briefly in \cite{DP2003}.

Let us now give a fuller summary of the paper.  
We begin in Section \ref{sect:prelim} by recalling the preliminary material we need, including the definitions of the partition category $\P$, the Brauer category $\B$ and the Temperley--Lieb category~$\TL$, as well as their twisted counterparts.  

Section \ref{sect:Pn} contains our main construction, which is a mapping
\[
\varphi = \varphi_K:\P\to\M(K):a\mt\ol a,
\]
where $K$ is any (fixed) unital semiring, and a partition $a\in \P_{m,n}$ is mapped to a ${2^m\times2^n}$ zero-one matrix $\ol a\in\M_{2^m,2^n}(K)$ determined by the `blocks' of $a$; see Definition~\ref{def:therep}.  This map~$\varphi$ is always injective, involutive and preserves the tensor operations; it is a morphism if and only if $K$ is an \emph{(additively) idempotent} semiring; see Proposition~\ref{prop:vp}, Theorem~\ref{thm:partition} and Remark \ref{rem:conv}. The proof of Theorem \ref{thm:partition} involves an analysis of what we call \emph{intertwining sets}; see Definition \ref{defn:is} and Lemma \ref{lem:intertwine}. It follows that for each fixed $n$ the partition monoid $\P_n$ embeds into the matrix monoid $\M_{2^n}(K)$ over any unital idempotent semiring $K$. We conclude Section \ref{sect:Pn} by showing that the dimensions of the matrices in the image of $\varphi$ are minimal with respect to faithful involutive tensor representations by matrices over \emph{any} semiring; see Theorem \ref{thm:min}.

In Section \ref{sect:twisted} we adapt our mapping $\varphi$ to construct a minimal faithful involutive tensor representation of the \emph{twisted} partition category $\P^\Phi$ over any semiring of characteristic~$0$, leading also to faithful representations of the twisted partition monoids $\P_n^\Phi$; see Theorem \ref{thm:twistpartition}.  This is possible because of the underlying combinatorial properties of the intertwining sets mentioned above, and to their connection with `float counting' parameters.
Theorem~\ref{thm:twistdpartition} provides a corresponding result for the \emph{$d$-twisted} categories~$\P^{\Phi,d}$ (and $d$-twisted partition monoids $\P_n^{\Phi,d}$), involving a similar construction over rings of characteristic~$2^{d+1}$.  When $K$ is a field of characteristic~$0$, we show in Theorem~\ref{thm:decomp} that the associated $\P^\Phi$-module $\V$ (consisting of disjoint copies of $2^n$-dimensional vector spaces over $K$) decomposes as a direct sum of two irreducible submodules, $\V=\V^+\op\V^-$, although neither $\V^+$ nor $\V^-$ is faithful; see Remark~\ref{rem:decomp}.

In Sections \ref{sect:Bn} and \ref{sect:TLn}, we drop the requirement that our representations respect the tensor structure and show how the above representation of $\P$ induces smaller degree faithful involutive representations of the subcategories $\B$ and $\TL$ (and hence of the submonoids $\B_n$ and $\TL_n$); see Theorems~\ref{thm:Brauer_odd},~\ref{thm:TL_even} and~\ref{thm:TLfib}.  In particular, a Temperley--Lieb diagram $a\in\TL_{m,n}$ is faithfully represented as an $f_m\times f_n$ matrix, where $f_n$ is the $n$th Fibonacci number.

Section \ref{sect:algebras} explores connections to representations of diagram \emph{algebras} and associated \emph{linear} diagram categories.  In particular, our representation of the twisted category $\P^\Phi$ leads to matrix representations of the linear partition categories $\P(K,2)$, and hence of the partition algebras $\P_n(K,2)$, where $K$ is a ring of characteristic $0$.  These are not faithful (see Example~\ref{ex:P2}), and neither are the restrictions to the Brauer algebras (see Example~\ref{ex:B3}).  However, the restrictions to the Temperley--Lieb algebras coincide with well-known representations \cite{B1937}, which \emph{are} known to be faithful when $K$ is a field \cite{J1994,DKP2006}.  We give a simple argument to deduce that our representation of the full linear Temperley--Lieb category $\TL(K,2)$ is itself faithful; see Theorem \ref{thm:TLC}.

Finally, Section \ref{sect:conclusion} contains a number of open questions for future research. A natural problem is to determine the \emph{minimum} dimensions of faithful linear representations of the (twisted) diagram categories and monoids.  In particular, it is intriguing to ask whether or when the faithful, involutive representations we give here are minimal.  For example, our representation of the partition monoid $\P_2$ over the Boolean semiring has degree $2^2 = 4$; it can be shown computationally with GAP \cite{GAP,Semigroups} that this monoid cannot be faithfully represented in degree $3$, but we have no theoretical method that would allow us to establish corresponding lower bounds in higher rank cases where computational approaches are infeasible.

\section{Preliminaries}\label{sect:prelim}

We now gather the definitions and results we need on general categories and semirings, as well as specific categories of matrices and diagrams/partitions.  For further details and background, see for example \cite{S2011,M1998,E2024,Howie1995,HV2019,ER2022c}.  We denote the sets of integers and natural numbers by $\ZZ = \{0,\pm1,\pm2,\ldots\}$ and $\NN = \{0,1,2,\ldots\}$, respectively.
For a natural number $n\in\NN$, we write $[n] = \{1,\ldots,n\}$, interpreting this to be empty when $n=0$. Throughout, maps will typically be written on the right and hence composed reading left to right.

\subsection{Categories}\label{subsect:cat}

All categories we consider are \emph{small}, meaning that their morphisms form a set. Formally, a category is a tuple $(\C,O,\bd,\br,\circ)$, where:
\bit
\item $\C$ is the set of morphisms,
\item $O$ is the set of objects,
\item $\bd:\C\to O$ and $\br:\C\to O$ are the domain and range maps, and
\item $\circ$ is the composition operation of morphisms, where $a\circ b$ is defined (for $a,b\in\C$) if and only if $\br(a)=\bd(b)$, in which case $\bd(a\circ b) = \bd(a)$ and $\br(a\circ b) = \br(b)$,
\eit
with composition satisfying the associative law where defined, and where for each object $s \in O$ there exists an \emph{identity morphism} at $s$ denoted $\io_s$ with $\bd(\io_s) = \br(\io_s) = s$ and such that $\io_s \circ a = a$ and $b \circ \io_s = b$ wherever defined.  We typically identify the category $(\C,O,\bd,\br,\circ)$ with the underlying (morphism) set $\C$, and use juxtaposition to denote composition: $ab = a\circ b$.  For objects $s,t\in O$, we write
\[
\C_{s,t} = \set{a\in\C}{\bd(a)=s ,\ \br(a) = t}
\]
for the set of all morphisms $s\to t$, and refer to such sets as hom-sets. Note that if $s$ and~$t$ are distinct there is no requirement that $\C_{s,t}$ be non-empty; for $s=t$ we have $\io_s \in \C_{s,s}$.  We write $\C_s = \C_{s,s}$ for the endomorphism monoid of the object $s\in O$.

An \emph{involution} on a category $\C$ is a map $\C\to\C:a\mt a^*$ satisfying the following for all $a,b \in \C$:
\bit
\item $\bd(a^*) = \br(a)$ and $\br(a^*) = \bd(a)$,
\item $a^{**} = a$,
\item $(a b)^* = b^* a^*$ when $\br(a) = \bd(b)$.
\eit
It is an easy exercise to deduce that the above axioms imply also
that $\io_s^* = \io_s$ for each object $s$.
An \emph{involutive category} is a category with a specified involution.  These are also called \emph{dagger categories} or \emph{$*$-categories} in the literature.  A \emph{regular $*$-category} is an involutive category $\C$ for which $a=aa^*a$ for all $a\in\C$; it follows then that also $a^*=a^*aa^*$, so that $a^*$ is a \emph{(generalised) inverse} of $a$.

The categories of interest in this paper are \emph{strict tensor categories}, meaning that the object and morphism sets have additional (totally defined) binary operations, satisfying natural conditions.  In all cases we consider, the object set $O$ of $\C$ will be either~$(\NN,+)$ or~$(\NN,\cdot)$, the additive or multiplicative monoids over the natural numbers~${\NN=\{0,1,2,\ldots\}}$, in which case the corresponding \emph{tensor} operation on $\C$ will be denoted~$\op$ or~$\ot$, respectively.
When $O=(\NN,+)$ the axioms are as follows:
\bit
\item $\bd(a\op b)=\bd(a)+\bd(b)$ and $\br(a\op b)=\br(a)+\br(b)$, 
\item $a\op(b\op c)=(a\op b)\op c$, 
\item $a\op \io_0=a=\io_0\op a$, 
\item $\io_m\op\io_n=\io_{m+n}$, 
\item $(a\circ b)\op(c\circ d) = (a\op c)\circ(b\op d)$ if $\br(a)=\bd(b)$ and $\br(c)=\bd(d)$. 
\eit
The axioms are analogous in the case $O=(\NN,\cdot)$; we replace $\op$ and $+$ by $\ot$ and $\cdot$ in each law, and $\io_0$ by $\io_1$ in the third.  The term `strict' refers to the fact that we have equality in the above laws, rather than isomorphism, as for more general tensor categories.  Since we only ever refer to strict tensor categories, we drop `strict' henceforth.

An \emph{involutive tensor category} is simultaneously an involutive and tensor category, in which $(a\op b)^* = a^*\op b^*$ or $(a\ot b)^* = a^*\ot b^*$, as appropriate.

A category \emph{representation} is a functor $\phi:\C\to\D$ between categories.  We say $\phi$ is:
\bit
\item \emph{faithful} if it is injective on objects and morphisms,
\item \emph{involutive} if $\C$ and $\D$ are both involutive categories, and $\phi$ preserves the involutions,
\item a \emph{tensor representation}, if $\C$ and $\D$ are both tensor categories, and $\phi$ preserves the tensor operations. 
\eit

\subsection{Semirings}\label{subsect:sr}

A \emph{(unital) semiring} is a tuple $(K,+,\cdot,0,1)$, such that:
\bit
\item $(K,+,0)$ is a commutative monoid, with $+$ referred to as addition,
\item $(K,\cdot,1)$ is a monoid, with $\cdot$ referred to as multiplication,
\item multiplication distributes over addition, and 
\item $0$ is a multiplicative zero element.
\eit
We typically identify a semiring with its underlying set $K$, and we always assume that~$K$ is non-trivial, in the sense that $0\not=1$.
We say that a semiring $K$:
\begin{itemize}
\item is \emph{commutative} if the multiplication of $K$ is commutative,
\item is \emph{(additively) idempotent} if $1+1=1$ (or equivalently if $\alpha+\alpha=\alpha$ for all $\alpha\in K$),
\item has \emph{characteristic 0} if $1$ generates an infinite subsemigroup under addition.
\end{itemize}
The prototypical example of an idempotent semiring is the \emph{Boolean semiring}, $\BB = \{0,1\}$, in which $1+1=1$ (and all other additions and multiplications follow from the axioms). Another important and widely studied example is the \textit{tropical} (or \textit{max-plus}) \textit{semiring} (see for example \cite{MS2015}). The natural numbers under ordinary addition and multiplication is a semiring of characteristic $0$.

In an arbitrary semiring of characteristic $0$, we shall identify a natural number $n$ with the element $\underbrace{1 + \cdots + 1}_n$ in $K$, interpreting this as the zero of $K$ when $n=0$. If $K$ is a semiring with the property that  $\underbrace{1+ \cdots + 1}_d = 0$ for some $d\geq 2$, then clearly $1$ (and hence, by distributivity, every element of $K$) has an additive inverse; thus in this case $K$ is a ring. If $d$ is the least such positive integer, then $K$ is a ring of characteristic $d$. In a ring of characteristic $d$, each integer $0\leq n<d$ will be identified with $\underbrace{1 + \cdots + 1}_n$~in~$K$.

\subsection{Categories of matrices}\label{subsect:M}

Throughout this section we fix a (non-trivial unital) semiring $K$.  For sets $I$ and $J$, we write $\M_{I,J}$ for the set of all $I\times J$ matrices over $K$; later we write $\M_{I,J}(K)$ when we wish to emphasise the semiring $K$, with an analogous convention for other sets of matrices.  For $a\in\M_{I,J}$, and for $i\in I$ and $j\in J$, we denote the $(i,j)$th entry of $a$ by $a_{i,j}$, and write $a = (a_{i,j})_{i\in I,j\in J}$, or simply $a=(a_{i,j})$ if the indexing sets are understood from context.  When $I$ or $J$ is the empty set, we assume that $\M_{I,J}$ consists of a unique $I\times J$ empty matrix.

If $I = [m] = \{1,\ldots,m\}$ and $J = [n] = \{1,\ldots,n\}$ for natural numbers $m,n\in\NN$, we simply write $\M_{m,n} = \M_{[m],[n]}$.  The set
\[
\M = \bigcup_{m,n\in\NN} \M_{m,n}
\]
is an involutive tensor category with object set $(\NN,\cdot)$, where composition and involution are given by ordinary matrix multiplication and transpose, the latter being denoted $a\mt a^T$.  The tensor operation is the \emph{Kronecker product} $\ot$, whose definition we recall below.  For $n\in\NN$, the endomorphism monoid $\M_n = \M_{n,n}$ is the full linear monoid of degree $n$ over $K$; its identity is the usual $n\times n$ identity matrix, which we denote by $I_n$.  Note that $\M_{n,0}$ and $\M_{0,n}$ respectively consist solely of the unique $n\times0$ or $0\times n$ empty matrix, for each $n\in\NN$.

In the special case that $K = \BB$ is the Boolean semiring, the category $\M(\BB)$ is isomorphic to the \emph{relation category} $\R$.  The latter has object set $\NN$, and the morphism set $\R_{m,n}$ consists of all relations $[m]\to[n]$, i.e.~all subsets of $[m]\times[n]$.  The rule for composing relations $\al\in\R_{m,n}$ and $\beta\in\R_{n,t}$ is
\[
\al\beta = \bigset{(i,j)\in[m]\times[t]}{(i,x)\in\al,\ (x,j)\in\beta\ (\exists x\in[n])}.
\]
The involution is given by $\al^* = \set{(y,x)}{(x,y)\in\al}$.
The above-mentioned isomorphism $\M(\BB) \to \R$ sends a matrix $a\in\M_{m,n}(\BB)$ to the relation ${\bigset{(i,j)\in[m]\times[n]}{a_{i,j}=1}}$.

\begin{defn}\label{defn:Kronecker}
For matrices $A=(a_{i,j})\in\M_{I,J}$ and $B=(b_{u,v})\in\M_{U,V}$, the \emph{Kronecker product} $A\ot B\in\M_{I\times U,J\times V}$ is the matrix with entry $a_{i,j}b_{u,v}$ in row-$(i,u)$, column-$(j,v)$.
\end{defn}

For example, if $I=[m]$, $J=[n]$, $U=[s]$ and $V=[t]$ for natural numbers ${m,n,s,t\in\NN}$, then the bijections
\begin{align*}
[m]\times[s]\to[ms]&:(i,u)\mt (i-1)s+u \\
\text{and}\qquad [n]\times[t]\to[nt]&:(j,v)\mt (j-1)t+v
\end{align*}
allow us to regard $A\ot B$ as in Definition \ref{defn:Kronecker} as the $ms\times nt$ matrix 
\[
A\ot B 
= \begin{pmatrix}
a_{1,1}B & \cdots & a_{1,n}B\\
\vdots&\ddots&\vdots\\
a_{m,1}B & \cdots & a_{m,n}B
\end{pmatrix} \in \M_{ms,nt}.
\]
(If any of $m$, $n$, $s$ or $t$ is $0$, then $A\ot B$ is an empty matrix of appropriate dimensions.)  It is this formulation of the Kronecker product $\ot$ that gives $\M$ its tensor category structure.

As another example, consider $A=(a_{X,Y})\in\M_{2^{[m]},2^{[n]}}$ and ${B=(b_{U,V})\in\M_{2^{[s]},2^{[t]}}}$ for natural numbers $m,n,s,t\in\NN$.  Here $2^{[m]} = \set{X}{X\sub[m]}$ denotes the power set of~$[m]$, and so on.  The bijections between power sets
\begin{align*}
2^{[m]}\times2^{[s]} \to 2^{[m+s]} &: (X,U) \mt X\cup(U+m) \\
\text{and}\qquad 2^{[n]}\times2^{[t]} \to 2^{[n+t]} &: (Y,V) \mt Y\cup(V+n) 
\end{align*}
allow us to regard $A\ot B$ as in Definition \ref{defn:Kronecker} as a $2^{[m+s]}\times 2^{[n+t]}$ matrix.  (In the above we write $U+m=\{u+m:u\in U\}$ and $V+n=\{v+n:v\in V\}$.)  For $P\sub[m+s]$ and $Q\sub[n+t]$, there exist unique $X\sub[m]$, $Y\sub[n]$, $U\sub[s]$ and $V\sub[t]$ such that  $P=X\cup(U+m)$ and $Q=Y\cup(V+n)$, giving
\begin{equation}\label{eq:ot}
(A\ot B)_{P,Q} = a_{X,Y}b_{U,V}.
\end{equation}

\subsection{Matrix representations and effective dimensions}\label{subsect:rep}

A matrix representation of a category $\C$ is a functor
\[
\phi:\C\to\M(K)
\]
for some semiring $K$. The main results of this paper concern matrix representations of tensor categories that are faithful, involutive, and in some cases respect the tensor structure.  The `efficiency' of such a representation can be measured by looking at the dimensions of the matrices in the image.  Let the object set of $\C$ be $O$, and for each $s\in O$ write~$d_s = s\phi$. Recall that, formally, we have defined the objects of $\M(K)$ to be natural numbers, so that each $d_s \in \NN$.  Note that $\io_s\phi = I_{d_s}$ for all $s\in O$.  It follows that $\phi$ maps a hom-set $\C_{s,t}$ into~$\M_{d_s,d_t}$.  In particular, $\phi$ maps each endomorphism monoid $\C_s$ to the monoid $\M_{d_s}$ of $d_s\times d_s$ matrices over $K$.  Consequently, if $\phi$ is faithful, then each~$\C_s$ embeds in $\M_{d_s}$, in which case we have upper bounds for the \emph{effective dimensions}:
\begin{equation}\label{eq:dim}
\dim_K(M) := \min\bigset{d}{M\text{ embeds in }\M_d(K)} \leq d_s \qquad\text{for any submonoid $M$ of $\C_s$.}
\end{equation}
Our main results below give faithful matrix representations of certain diagram categories (to be defined shortly).  For each category $\C$ considered we also state the dimension of the image of each endomorphism monoid $\C_s$, which leads to upper bounds on effective dimensions of (submonoids) of the corresponding diagram monoids, as in \eqref{eq:dim}. See in particular Sections~\ref{subsect:min} and~\ref{sect:conclusion} below for further discussion.

\subsection{Diagram categories and monoids}\label{subsect:P}

For $A\sub\NN$ we write ${A' = \set{x'}{x\in A}}$ and $A'' = \set{x''}{x\in A}$ for two distinguished disjoint copies of $A$.

\begin{defn}
\label{defn:P}
For natural numbers $m,n\in\NN$ we write $\P_{m,n}$ for the set of all set partitions of $[m]\cup[n]'$.  The \emph{partition category} has object set $\NN$, and morphism set
\[
\P = \bigcup_{m,n\in\NN}\P_{m,n},
\]
under an operation that will be defined shortly.  

For a partition $a\in\P_{m,n}$ we write $\bd(a) = m$ and $\br(a) = n$.  We refer to the elements of~$a$ as the \emph{blocks} of $a$, or as \emph{$a$-blocks}.  We identify $a$ with any graph on vertex set $[m]\cup[n]'$ whose connected components are the $a$-blocks.  Such graphs are drawn in the plane $\mathbb R^2$ with vertex $x\in[m]$ in position $(x,1)$, and vertex $x'\in[n]'$ in position $(x,0)$; thus, the vertices from $[m]$ form a single upper row ordered consecutively from left to right, and likewise the vertices from $[n]'$ form a single lower row.  All such graphs are drawn with all edges contained within the region
\begin{equation}\label{eq:region}
\set{(x,y)\in\mathbb R^2}{0\leq y\leq 1}.
\end{equation}
(This is primarily important for planarity considerations below.)  See Figure \ref{fig:P6} for some examples.

The composition/product of partitions $a,b\in\P$ is defined if and only if $\br(a)=\bd(b)$.  So consider composable partitions $a\in\P_{m,n}$ and $b\in\P_{n,t}$, for some $m,n,t\in\NN$.  Informally,~$ab$ is calculated by first `glueing' the bottom of $a$ to the top of $b$, and then determining the connections among the top and bottom rows of vertices in this graph.
To formalise this, we first define three further graphs:
\begin{itemize}
\item $a^\downarrow$, with vertex set $[m]\cup[n]''$, obtained from $a$ by changing each lower vertex $x'$ to~$x''$,
\item $b^\uparrow$, with vertex set $[n]''\cup[t]'$, obtained from $b$ by changing each upper vertex $x$ to~$x''$, and
\item $\Pi(a,b)$, with vertex set $[m]\cup[n]''\cup[t]'$, whose edge set is the union of the edge sets of $a^\downarrow$ and $b^\uparrow$.
\end{itemize}
The graph $\Pi(a,b)$ is called the product graph of $a$ and $b$; it is drawn with vertices from~$[n]''$ in a middle row (on the line $y=\frac12$).  The product $ab\in\P_{m,t}$ is then defined so that ${x,y\in[m]\cup[t]'}$ belong to the same $ab$-block if and only if $x$ and $y$ are contained in the same connected component of $\Pi(a,b)$.  A sample product is considered in Example~\ref{ex:P6}.

The endomorphism monoids in $\P$ are the \emph{partition monoids} $\P_n = \P_{n,n}$ for $n\in\NN$.  The identity element of~$\P_n$ is the partition
\[
\io_n = \big\{\{1,1'\}, \{2,2'\}, \ldots, \{n,n'\}\big\} = \custpartn{1,2,4}{1,2,4}{\stline11\stline22\stline44\udotted24\ddotted24}.
\]
(Note that $\io_0$ is the empty partition, the unique element of $\P_0$.)  The group of units of~$\P_n$ (i.e.~the automorphism group at the object $n$) is the symmetric group $\S_n$.  Here a permutation $\pi\in\S_n$ is identified with the partition ${\bigset{\{i,(i\pi)'\}}{i\in[n]}\in\P_n}$.
\end{defn}

\begin{defn}
\label{defn:Phi}
Consider a pair of composable partitions ${a\in\P_{m,n}}$ and $b \in \P_{n,t}$.  The product graph $\Pi(a, b)$ may contain connected components that lie entirely in $[n]''$; viewed as subsets of vertices these components do not depend upon the choice of graphical representation of $a$ and $b$. We refer to these as \textit{floating components}, and write $\Phi(a,b)$ to denote the number of them.
\end{defn}

\begin{ex}\label{ex:P6}
Computation of the product of 
\begin{align*}
a &= \big\{\{1,4\}, \{2,3,4',5'\}, \{1',2',6'\}, \{3'\}\big\} \in\P_{4,6} \\
\text{and}\qquad b &= \big\{\{1,2\}, \{3,4,1'\}, \{5,4',5'\}, \{6\}, \{2',3'\}\big\} \in \P_{6,5}
\intertext{is  illustrated in  Figure \ref{fig:P6} using graphical representations of $a$ and $b$. The product graph~$\Pi(a,b)$ is constructed by identifying each vertex $i'$ in $a$ with the vertex~$i$ in $b$, and relabelling this as $i''$.  The product}
ab &= \big\{\{1,4\},\{2,3,1',4',5'\},\{2',3'\}\big\} \in \P_{4,5}
\end{align*}
has blocks corresponding to the non-empty intersections of the connected components of $\Pi(a,b)$ with $[4] \cup [5]'$. Notice that one of the connected components of $\Pi(a,b)$, namely $\{1'',2'',6''\}$, lies within~$[6]''$; this is a floating component, and since it is the only such component we have~$\Phi(a,b)=1$.
\end{ex}

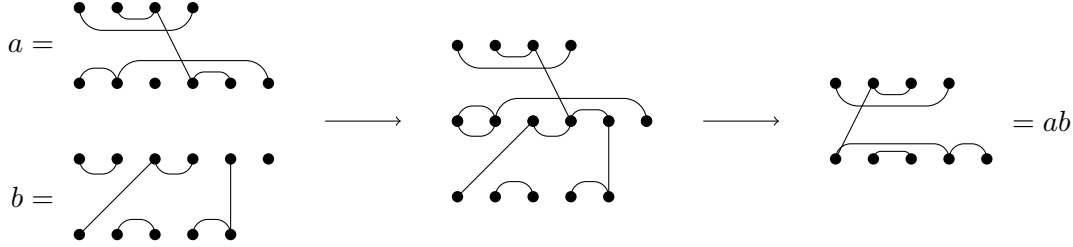
\begin{figure}[ht]
\begin{center}
\begin{tikzpicture}[scale=.5]

\begin{scope}[shift={(0,0)}]	
\uvs{1,...,4}
\lvs{1,...,6}
\uarcx14{.6}
\uarcx23{.3}
\darc12
\darcx26{.6}
\darcx45{.3}
\stline34
\draw(0.6,1)node[left]{$a=$};
\draw[->](7.5,-1)--(9.5,-1);
\end{scope}

\begin{scope}[shift={(0,-4)}]	
\uvs{1,...,6}
\lvs{1,...,5}
\uarc12
\uarc34
\darc45
\darc23
\stline31
\stline55
\draw(0.6,1)node[left]{$b=$};
\end{scope}

\begin{scope}[shift={(10,-1)}]	
\uvs{1,...,4}
\lvs{1,...,6}
\uarcx14{.6}
\uarcx23{.3}
\darc12
\darcx26{.6}
\darcx45{.3}
\stline34
\draw[->](7.5,0)--(9.5,0);
\end{scope}

\begin{scope}[shift={(10,-3)}]	
\uvs{1,...,6}
\lvs{1,...,5}
\uarc12
\uarc34
\darc45
\stline31
\stline55
\darc23
\end{scope}

\begin{scope}[shift={(20,-2)}]	
\uvs{1,...,4}
\lvs{1,...,5}
\uarcx14{.6}
\uarcx23{.3}
\darc14
\darc45
\stline21
\darcx23{.2}
\draw(5.4,1)node[right]{$=ab$};
\end{scope}

\end{tikzpicture}
\caption{Partitions $a\in\P_{4,6}$ and $b\in \P_{6,5}$ and their product $ab\in\P_{4,5}$, with the product graph $\Pi(a,b)$ in the middle.  This has one floating component, namely $\{1'',2'',6''\}$, so $\Phi(a,b)=1$ in this case.}
\label{fig:P6}
\end{center}
\end{figure}

\begin{defn}
\label{def:block}
The blocks of a partition $a\in\P_{m,n}$ can in general be of three types.  We say that $X \in a$ is:
\begin{itemize}
    \item an \textit{upper $a$-block} if $X \subseteq [m]$,
    \item a \textit{lower $a$-block} if $X \subseteq [n]'$,
    \item a \textit{transversal $a$-block} (or just a \textit{transversal}) otherwise, i.e.~if $X$ contains at least one element of $[m]$ and at least one element of $[n]'$.
\end{itemize}
We write $a = \begin{partn}{6} A_1&\cdots&A_r&C_1&\cdots&C_s\\ \hhline{~|~|~|-|-|-} B_1&\cdots&B_r&D_1&\cdots&D_t\end{partn}$ to indicate that $a$ has transversals $A_i\cup B_i'$ ($1\leq i\leq r$), upper blocks $C_i$ ($1\leq i\leq s$) and lower blocks $D_i'$ ($1\leq i\leq t$).  Note that any of $r$, $s$ or $t$ could be $0$; we have $r=s=t=0$ if and only if $m=n=0$, in which case $a=\es$ is the unique element of $\P_0$.
\end{defn}

\begin{defn}
The (co)domain, (co)kernel and rank of a partition $a \in \P_{m,n}$ are defined by:
\begin{itemize}
\item ${\rm dom}(a) = \{x\in[m]:x \text{ is contained in a transversal }a\mbox{-block}\}$,
\item ${\rm codom}(a)= \{x\in[n]:x' \text{ is contained in a transversal }a\mbox{-block}\}$,
\item ${\rm ker}(a) = \{(x,y) \in [m] \times [m]: x \mbox{ and } y \mbox{ are in the same } a\mbox{-block}\}$,
\item ${\rm coker}(a) = \{(x,y) \in [n] \times [n]: x' \mbox{ and } y' \mbox{ are in the same } a\mbox{-block}\}$,
\item ${\rm rank}(a)$  is the number of transversal $a$-blocks.
\end{itemize}
Note that $\dom(a)$ and $\codom(a)$ are respectively subsets of $[m]$ and $[n]$; $\ker(a)$ and $\coker(a)$ are respectively equivalences on~$[m]$ and~$[n]$; and $0\leq\rank(a)\leq \min(m,n)$ is an integer.  If $a = \begin{partn}{6} A_1&\cdots&A_r&C_1&\cdots&C_s\\ \hhline{~|~|~|-|-|-} B_1&\cdots&B_r&D_1&\cdots&D_t\end{partn}$, then
\begin{align*}
\dom(a) &= A_1\cup\cdots\cup A_r, & [m]/\ker(a) &= \{A_1,\ldots,A_r,C_1,\ldots,C_s\}, & \rank(a) &= r,\\
\codom(a) &= B_1\cup\cdots\cup B_r, & [n]/\coker(a) &= \{B_1,\ldots,B_r,D_1,\ldots,D_t\}.
\end{align*}
\end{defn}

\begin{ex}
The partition $a \in \P_{4,6}$ illustrated in Figure \ref{fig:P6} has one upper block, two lower blocks and one transversal. Thus ${\rm rank}(a)=1$. Moreover, ${\rm dom}(a) = \{2,3\}$, ${\rm codom}(a) = \{4,5\}$, ${\rm ker}(a)$ is the equivalence relation on $[4]$ with classes $\{1,4\}$ and $\{2,3\}$, and ${\rm coker}(a)$ is the equivalence relation on $[6]$ with classes $\{1,2,6\}$, $\{3\}$ and $\{4,5\}$.
\end{ex}

\begin{defn}
The partition category $\P$ is a regular $*$-category, with respect to the involution ${}^*$ given by applying the relabelling $i \leftrightarrow i'$.  This corresponds to a vertical reflection of a graphical representative.  For example, with $a\in \P_{4,6}$ as in Figure~\ref{fig:P6}, we have $a^* = \custpartn{1,...,6}{1,...,4}{\darcx14{.6}\darcx23{.3}\uarc12\uarcx26{.6}\uarcx45{.3}\stline43} \in \P_{6,4}$.
\end{defn}

\begin{defn}
The partition category $\P$ is a tensor category, under a horizontal stacking operation denoted by $\op$.  Here for $a\in\P_{m,n}$ and $b\in\P_{s,t}$ we define $a\op b\in\P_{m+s,n+t}$ to be the result of placing $a$ and $b$ side-by-side, with $a$ on the left.  For example,
\[
\custpartn{1,2,3,4}{1,2,3}{\stline11\stline23\uarc12\uarc34} \op \custpartn{1,2,3,4}{1,2,3,4,5,6}{\stline44\uarc12\uarc23\uarc34\darc12\darc45\darc56} = \custpartn{1,...,8}{1,...,9}{\stline11\stline23\uarc12\uarc34 \stline87\uarc56\uarc67\uarc78\darc45\darc78\darc89}
\]
Formally, every $a$-block is an $(a\op b)$-block, while every $b$-block $X\cup Y'$ (with $X$ or $Y$ possibly empty) gives rise to the $(a\op b)$-block $(X+m)\cup(Y+n)'$.
\end{defn}

\begin{defn}
The \textit{Brauer category}~$\B$ is the subcategory of $\P$ consisting of those partitions in which each block has size~$2$. The \textit{Temperley--Lieb category} $\TL$ is the subcategory of~$\B$ consisting of all planar Brauer partitions, i.e.~those that can be represented by a graph with no crossings, remembering that all edges must be drawn within the region specified in \eqref{eq:region}.  For example, we have
\[
\custpartn{1,...,10}{1,...,8}{\stline13 \stline88 \uarc34 \uarc56\uarc9{10} \uarcx27{.7} \darc12 \darc56 \darcx47{.7}} \in \TL, \qquad\text{while}\qquad
\custpartn{1,...,10}{1,...,8}{\stline23 \stline88 \uarc34 \uarc56\uarc9{10} \uarcx17{.7} \darc12 \darc56 \darcx47{.7}} \in \B \setminus \TL.
\]
Since $\B$ and $\TL$ are closed under both ${}^*$ and $\op$, both are involutive tensor categories.
Endomorphism monoids in $\B$ and $\TL$ are Brauer monoids $\B_n$ and Temperley--Lieb monoids~$\TL_n$, respectively.  Temperley--Lieb monoids are also known as Jones monoids in the literature.

Note that for $m,n\in\NN$ we have $\B_{m,n}\not=\emptyset \iff \TL_{m,n}\not=\emptyset \iff m\equiv n \bmod2$.  Thus, $\B$ and $\TL$ both decompose as the disjoint unions of two subcategories:
\[
\B = \B_\even\sqcup\B_\odd \AND \TL = \TL_\even\sqcup\TL_\odd,
\]
where
\[
\B_\even = \bigcup_{m,n\in2\NN}\B_{m,n} \AND \B_\odd = \bigcup_{m,n\in2\NN+1}\B_{m,n},
\]
with analogous meanings for $\TL_\even$ and $\TL_\odd$.  Note that $\B_\even$ and $\TL_\even$ are closed under $\op$, but $\B_\odd$ and $\TL_\odd$ are not; all four of these subcategories are closed under ${}^*$.
\end{defn}

\subsection{Twisted diagram categories and monoids}\label{subsect:t}

Twisted diagram monoids incorporate the product of an underlying diagram monoid (such as $\P_n$, $\B_n$ or $\TL_n$) as well as the numbers $\Phi(a,b)$ from Definition \ref{defn:Phi}.  The definition extends without serious modification to twisted diagram \emph{categories}:

\begin{defn}
\label{defn:PPhi}
The \textit{twisted partition category} is the involutive tensor category with object set $\NN$, morphism set
\[
\P^\Phi = \NN\times\P = \{(i,a): i \in \mathbb{N},\ a \in \P\},
\]
and operations given by:  
\bit
\item $\bd(i,a) = \bd(a)$ and $\br(i,a) = \br(a)$,
\item $(i,a)(j,b) = (i+j+\Phi(a,b), ab)$ when $\br(a)=\bd(b)$,
\item $(i,a)^* = (i,a^*)$,
\item $(i,a)\op(j,b) = (i+j,a\op b)$.
\eit
Associativity of the composition follows from the `twisting identity'
\[
\Phi(a,b)+ \Phi(ab,c) = \Phi(a,bc)+ \Phi(b,c) \qquad\text{for composable $a,b,c\in\P$,}
\]
as observed in \cite{W2007,FL2011}.  This identity was proved in the monoid case in \cite[Lemma 4.1]{FL2011}, but the proof works without significant modification for the category $\P$. 
Hom-sets in $\P^\Phi$ have the form $\P_{m,n}^\Phi = \NN \times \P_{m,n}$,
and endomorphism monoids are the \emph{twisted partition monoids} $\P_n^\Phi = \P_{n,n}^\Phi$.  The identity of $\P_n^\Phi$ is $(0,\io_n)$.
\end{defn}

\begin{defn}
\label{defn:PPhid}
For each natural number $d\in\NN$, the \textit{$d$-twisted partition category} $\P^{\Phi,d}$ is the involutive tensor category with object set $\NN$, and (finite) morphism sets
\[
\P_{m,n}^{\Phi,d} =  \big(\{0,1,\ldots,d\}\times\P_{m,n}\big) \cup \{\bz_{m,n}\}.
\]
The operations are induced from those of $\P^\Phi$, with the morphisms $\bz_{m,n}$ acting as zero elements that also occur when compositions or tensor sums of non-zero morphisms yield integer-coordinates greater than $d$.  So we have:
\begin{align*}
(i,a)(j,b) &= \begin{cases}
(i+j+\Phi(a,b), ab) &\text{if $\br(a)=\bd(b)$ and $i+j+\Phi(a,b) \leq d$,}\\
\bz_{\bd(a), \br(b)} &\text{if $\br(a)=\bd(b)$ and $i+j+\Phi(a,b) > d$,}
\end{cases}\\[2mm]
(i,a)\op(j,b) &= \begin{cases}
(i+j,a\op b) &\text{if $i+j\leq d$,}\\
\bz_{\bd(a)+\bd(b),\br(a)+\br(b)} &\text{if $i+j> d$,}
\end{cases}
\end{align*}
and also
\[
\bz_{m,n}\circ\P_{n,t}^{\Phi,d} = \bz_{m,t} = \P_{m,n}^{\Phi,d}\circ\bz_{n,t} \AND \bz_{m,n}\op\P_{s,t}^{\Phi,d} = \bz_{m+s,n+t} = \P_{m,n}^{\Phi,d}\op\bz_{s,t}.
\]
The involution is given by
\[
(i,a)^* = (i,a^*) \AND \bz_{m,n}^*=\bz_{n,m}.
\]
One can think of $\P^{\Phi,d}$ as a \emph{Rees quotient} of $\P^\Phi$ by the ideal $\{d+1,d+2,\ldots\} \times \P$.  (For more details, and for the meaning of these terms, see \cite{ER2023,ER2022c}.)

Endomorphism monoids $\P_n^{\Phi,d}$ are the $d$-twisted partition monoids; they were denoted~$\P_{n,d}^\Phi$ in \cite{ER2022b,ER2022c}, but here we have placed the `$d$' in a superscript to avoid confusion with hom-sets $\P_{m,n}^\Phi$ in the category $\P^\Phi$.
\end{defn}

\begin{defn}
\label{defn:BTLPhi}
One also of course has ($d$-)twisted Brauer and Temperley--Lieb categories, which we denote by $\B^\Phi$, $\B^{\Phi,d}$, $\TL^\Phi$ and $\TL^{\Phi,d}$ with the obvious meanings.  Note that in the $d$-twisted categories $\B^{\Phi,d}$ and $\TL^{\Phi,d}$ we only need zero morphisms $\bz_{m,n}$ when $m\equiv n\bmod2$.
Endomorphism monoids $\TL_n^\Phi$ are known in the literature as Kauffman monoids \mbox{\cite{BDP2002,ACHLV2015,LF2006,KV2020,DE2017}}.
\end{defn}

\begin{rem}
\label{rem:Z}
One can replace $\NN$ by $\ZZ$ in the definitions of $\P^\Phi$, $\B^\Phi$ and $\TL^\Phi$, to obtain twisted categories on underlying sets $\ZZ\times\P$, $\ZZ\times\B$ and $\ZZ\times\TL$.  This has proved fruitful in the monoid case in a number of studies \cite{KV2020,EGPR2025,AV2020}.
\end{rem}

\section{Representing the partition category}\label{sect:Pn}

This section is central to the rest of the paper, and contains our main result, Theorem~\ref{thm:partition}, which gives a faithful involutive tensor representation of the partition category~$\P$ in the matrix category $\M$ over any (non-trivial) idempotent semiring.  Sections~\ref{subsect:01} and~\ref{subsect:IS} contain the relevant definitions and some preliminary lemmas, while Section~\ref{subsect:main} contains the theorem and its proof.  Finally, in Section \ref{subsect:min} we show that the dimensions of the matrices in the image of our representation are minimal for \emph{any} faithful involutive tensor representation $\P\to\M(K)$ for \emph{any} semiring $K$.

\subsection{A mapping to zero-one matrices}\label{subsect:01}

Fix a (non-trivial unital) semiring $K$, and let ${\M = \M(K)}$ be the involutive tensor category of matrices over $K$.  In all that follows, we shall view elements of $\M_{2^m,2^n} = \M_{2^m,2^n}(K)$ as matrices with rows and columns indexed by the subsets of $[m]$ and $[n]$, respectively.  Orderings on the rows and columns are not needed unless explicitly writing down matrices, in which case we shall be careful to specify such orders.  The following definition will be crucial.
In all that follows, when we say that a set is a union of sets from a particular collection of sets, we include the possibility that this union is empty.

\begin{defn}
\label{def:therep}
Given a (non-trivial) semiring $K$, we define the mapping
\[
\varphi = \varphi_K: \P \rightarrow \M(K):a\mapsto \overline a,
\]
where for each $a \in \P_{m,n}$ the matrix $\overline{a}\in \M_{2^m,2^n}(K)$ is defined, for $X\sub[m]$ and $Y\sub[n]$, by
\[
\overline{a}_{X,Y} = \begin{cases} 1 & \mbox{ if } X \cup Y' \mbox{ is a (possibly empty) union of } a\mbox{-blocks},\\
0 & \mbox{ else.}
\end{cases}
\]

\end{defn}

\begin{ex}\label{ex:P2_0}
Consider the following partitions, all from $\P_2$:
\[
a = \custpartn{1,2}{1,2}{\stline11\stline21\uarc12} \COMMA
b = \custpartn{1,2}{1,2}{\stline12\stline22\uarc12} \AND
c = \custpartn{1,2}{1,2}{\darc12} .
\]
Indexing the subsets of $[2]=\{1,2\}$ in the order $\es$, $\{1\}$, $\{2\}$, $\{1,2\}$, one can check that
\begin{equation}\label{eq:P2}
\ol a = \begin{pmatrix}
\fourrow1010\\
\fourrow0000\\
\fourrow0000\\
\fourrow0101
\end{pmatrix} \COMMa
\ol b = \begin{pmatrix}
\fourrow1100\\
\fourrow0000\\
\fourrow0000\\
\fourrow0011
\end{pmatrix} \ANd
\ol c = \begin{pmatrix}
\fourrow1001\\
\fourrow1001\\
\fourrow1001\\
\fourrow1001
\end{pmatrix} .
\end{equation}
We will return to the matrices of this example later.
\end{ex}

The mapping $\varphi:\P\to\M:a\mt\ol a$ is not necessarily a category representation in general; as we will see in the next section, preservation of composition depends on the structure of the underlying semiring $K$.  Before we get to this, however, we establish some basic properties of $\varphi$ that hold for arbitrary $K$, namely that it is injective, maps identity partitions to identity matrices, and preserves the involutions and tensor operations.

\newpage

\begin{prop}\label{prop:vp}
For all $n\in\NN$ and $a,b\in\P$ we have:
\ben
\item \label{vp1} $\ol a=\ol b \implies a=b$,
\item \label{vp2} $\ol{\io_n} = I_{2^n}$,
\item \label{vp3} $\ol{a^*} = \ol a^T$,
\item \label{vp4} $\ol{a\op b} = \ol a \ot \ol b$,
\een
where for \ref{vp4} we recall the indexing convention for Kronecker products from \eqref{eq:ot} above.
\end{prop}

\begin{proof}
\ref{vp1}  This follows from the fact that the blocks of a partition $a\in \P$ can be recovered from the matrix $\overline{a}\in\M$. Indeed, 
if $ a\in \P_{m,n}$, then the $a$-blocks are precisely the minimal (with respect to inclusion) non-empty sets in $\{X\cup Y':X\sub[m],\ Y\sub[n],\ \ol a_{X,Y}=1\}$.

\ref{vp2}  For $X,Y\sub[n]$ it is easy to see that $X \cup Y'$ is a union of $\io_n$-blocks if and only if $X = Y$, and the claim follows.

\ref{vp3}  This follows from the obvious fact that $X\cup Y'$ is an $a$-block if and only if $Y\cup X'$ is an $a^*$-block.

\ref{vp4}  Suppose $a\in\P_{m,n}$ and $b\in\P_{s,t}$, so that $a\op b\in\P_{m+s,n+t}$.  Let $P\sub[m+s]$ and $Q\sub[n+t]$, and write $P = X\cup (U+m)$ and $Q = Y\cup (V+n)$, where $X\sub[m]$, $Y\sub[n]$, $U\sub[s]$ and $V\sub[t]$.  Then keeping \eqref{eq:ot} in mind, we have
\begin{align*}
(\ol{a\op b})_{P,Q} = 1 &\iff P\cup Q' \text{ is a union of $(a\op b)$-blocks}\\
&\iff X\cup Y' \text{ is a union of $a$-blocks and } U\cup V' \text{ is a union of $b$-blocks}\\
&\iff \ol a_{X,Y} = 1 \text{ and } \ol b_{U,V} = 1 \\
&\iff \ol a_{X,Y} \ol b_{U,V} =1 \\
&\iff (\ol a \ot \ol b)_{P,Q} =1 .  \qedhere
\end{align*}
\end{proof}

\begin{rem}
The map $\varphi$ preserves another natural structure on $\P$, in the form of an order studied (in the monoid case) for example in \cite{ER2022a,FitzGerald2013}.  For $a,b\in\P_{m,n}$ we write $a\preceq b$ if $a$ \emph{refines}~$b$, in the sense that every $a$-block is contained in a $b$-block, or equivalently that every $b$-block is a union of $a$-blocks.  This is the same as saying that the equivalence relation on $[m]\cup[n]'$ associated to $a$ is contained in that associated to~$b$.  If $ a \preceq b$ then~$\ol b_{X,Y}=1$ clearly implies $\ol a_{X,Y}=1$ for $X\sub[m]$ and $Y\sub[n]$. Conversely, if for all subsets $X\sub[m]$ and $Y\sub[n]$ we have $\ol b_{X,Y}=1 \implies \ol a_{X,Y}=1$, then  by choosing $X$ and~$Y$ so that $X\cup Y'$ is a $b$-block we find that each 
$b$-block must be a union $a$-blocks. Thus, writing $A\preceq B$ for zero-one matrices $A$ and $B$ of the same dimensions to mean that $A_{i,j}\leq B_{i,j}$ for all $i,j$ (where $0<1$), we then have
\[
a\preceq b \quad \Leftrightarrow\quad \ol a \succeq \ol b \qquad\text{for all }a,b\in\P.
\]
(Incidentally, this also implies that $\varphi$ is injective.)
When $K=\mathbb B$ is the Boolean semiring, the $\preceq$ ordering on matrices corresponds to subset inclusion of the associated binary relations.  Thus, identifying $\M(\mathbb B)$ with the relation category $\R$, we have
\[
a\preceq b \quad\Leftrightarrow\quad \ol a \supseteq \ol b \qquad\text{for all } a,b\in\P.
\]
\end{rem}

\subsection{Intertwining sets}\label{subsect:IS}

As we have already said, the behaviour of the matrices~$\overline{a}$ under multiplication depends on the structure of the underlying semiring $K$.  To address this, we require the notion of an \emph{intertwining set}.  Working towards the definition, we begin by introducing some convenient notation and terminology.  For $a\in\P_{m,n}$, and for $X \sub[m]$ and $Y\sub[n]$, we write
\begin{align*}
Xa &= \{j \in [n]: \text{there exists $x \in X$ such that $x$ and $j'$ are in the same $a$-block}\},\\
aY &= \{j \in [m]: \text{there exists $y \in Y$ such that $j$ and $y'$ are in the same $a$-block}\}.
\end{align*}
One can think of $Xa$ as the `image' of $X$ under $a$, and $aY$ as the `preimage' of $Y$.

\begin{lem}
\label{lem:ones}
For $a \in \P_{m,n}$, and for $X\sub[m]$ and $Y \subseteq [n]$, the following are equivalent:
\ben
\item \label{ones1} $\overline{a}_{X,Y} = 1$,
\item \label{ones2} $X$ is a union of ${\rm ker}(a)$-classes and $Y = Xa \cup P$ where $P'$ is a union of lower $a$-blocks,
\item \label{ones3} $Y$ is a union of ${\rm coker}(a)$-classes and $X = aY \cup Q$ where $Q$ is a union of upper $a$-blocks.
\een
\end{lem}

\begin{proof}
\ref{ones1}$\implies$\ref{ones2}
If $X \cup Y'$ is a union of $a$-blocks, say of transversals $A_1\cup B_1',\ldots,A_r\cup B_r'$, upper blocks $C_1,\ldots,C_s$ and lower blocks $D_1',\ldots,D_t'$, then $X = A_1\cup\cdots\cup A_r\cup C_1\cup\cdots\cup C_s$, $Y = B_1\cup\cdots\cup B_r\cup D_1\cup\cdots\cup D_t$ and $Xa=B_1\cup\cdots\cup B_r$, and \ref{ones2} quickly follows.

The implication \ref{ones2}$\implies$\ref{ones1} is clear.  The equivalence \ref{ones1}$\iff$\ref{ones3} is dual.
\end{proof}

\begin{defn}
\label{defn:is}
Consider a composable pair of partitions $a\in\P_{m,n}$ and $b\in\P_{n,t}$, and a pair of subsets $X\sub[m]$ and $Y\sub[t]$.  We say that a subset $Z \subseteq [n]$ is an \textit{intertwining set} for $(a,b)$ with respect to $(X,Y)$ if $X \cup Z'$ is a union of $a$-blocks and $Z \cup Y'$ is a union of $b$-blocks, i.e.~if $\ol a_{X,Z} = 1 = \ol b_{Z,Y}$.
We note that for such a $Z$ to exist, $X$ must be a union of $\ker(a)$-classes, and $Y$ a union of $\coker(b)$-classes.
\end{defn}

\begin{ex}
For $a\in \P_{4,6}$ and $b\in \P_{6,5}$ from Figure \ref{fig:P6}, $X=[4]=\{1,2,3,4\}$ is a union of $\ker(a)$-classes and $Y=\{1,4,5\}$ a union of $\coker(b)$-classes. Moreover, one can check that the intertwining sets are precisely $\{3,4,5\}$ and $[6] = \{1,2,3,4,5,6\}$.
\end{ex}

The next result is immediate from the definitions.

\begin{lem}
\label{lem:prod}
Let $a\in\P_{m,n}$ and $b \in \P_{n,t}$.  Then for any $X\sub[m]$ and $Y\sub[t]$ we have
\[
(\ol a\, \ol b)_{X,Y} = \sum_{Z \subseteq [n]} \overline{a}_{X,Z}\overline{b}_{Z,Y} = 
\underbrace{1 + \cdots + 1}_{k}\in K,
\]
where $k$ is the number of interwining sets for $(a,b)$ with respect to $(X,Y)$.
\end{lem}

It should now be clear from Lemma \ref{lem:prod} that to determine whether $\varphi=\varphi_K$ is a morphism requires some combinatorics. Specifically, we need to count (in the additive operation of $K$) the number of intertwining sets for pairs of partitions and sets.  We do so in the next lemma, which involves the twisting parameters $\Phi(a,b)$ from Definition \ref{defn:Phi}.  Certain notation used in the proof will be used again later in the paper, so we establish this before we give the statement.

\begin{defn}\label{defn:Z0}
Consider a composable pair of partitions $a\in\P_{m,n}$ and $b\in\P_{n,t}$, and a pair of subsets $X\sub[m]$ and $Y\sub[n]$.  Define $\sim$ to be the join of the equivalence relations $\coker(a)$ and $\ker(b)$ on $[n]$, i.e.~the transitive closure of $\coker(a) \cup \ker(b)$.  
The significance of this relation is that if $C$ is a connected component of the product graph $\Pi(a,b)$, then $C\cap[n]''$ is either empty, or is equal to $D''$ for some $\sim$-class $D$.
For $j \in [n]$ let $[j]_\sim$ denote the $\sim$-class containing $j$, and define 
\[
Z_0 = Z_0(a,b,X,Y) : = \bigcup_{j \in Xa \cup bY} [j]_\sim.
\]
\end{defn}

\begin{lem}
\label{lem:intertwine}
With the notation of Definition \ref{defn:Z0}, and for $Z\subseteq[n]$, the following are equivalent:
\ben
    \item \label{intertwine1} $Z$ is an intertwining set for $(a,b)$ with respect to $(X,Y)$,
    \item \label{intertwine2} $X \cup Y'$ is a union of $ab$-blocks \emph{(}i.e.~$\ol{ab}_{X,Y}=1$\emph{)} and $Z=Z_0 \cup F$ where $F''$ is a union of floating components of the product graph $\Pi(a,b)$.
\een
Consequently, the number of intertwining sets for $(a,b)$ with respect to $(X, Y)$ is $2^{\Phi(a,b)}$ if $X \cup Y'$ is a union of $ab$-blocks, or $0$ otherwise.
\end{lem}

\begin{proof}
\ref{intertwine2} $\Rightarrow$ \ref{intertwine1}.
Suppose first that $X \cup Y'$ is a union of $ab$-blocks and $Z = Z_0 \cup F$ where~$F''$ is a union of floating components of $\Pi(a,b)$. In particular, $X$ is a union of ${\rm ker}(ab)$-classes and~$Y$ is a union of ${\rm coker}(ab)$-classes. We need to show that $X\cup Z'$ is a union of $a$-blocks and $Z \cup Y'$ is a union of $b$-blocks. Since $F'$ is a union of lower $a$-blocks and $F$ a union of upper $b$-blocks, it suffices to show that $X\cup Z_0'$ is a union of $a$-blocks and $Z_0 \cup Y'$ is a union of $b$-blocks. We show the first of these two statements, the second being dual.

Suppose then that $u \in X \cup Z_0'$, and that $v$ is in the same $a$-block as $u$. We need to show that $v \in X \cup Z_0'$ too. The proof now splits into four cases, depending on whether each of $u$ and $v$ lies in~$[m]$ or~$[n]'$.

\medskip
Case 1:  $u,v \in [m]$. Since $u\in X \cup Z_0'$ we must have $u \in X$, and since $u$ and $v$ lie in the same $a$-block we also have $(u,v) \in {\rm ker}(a)$. Since $X$ is a union of $\ker(ab)$-classes, and hence of $\ker(a)$-classes, it follows that $v \in X\sub X\cup Z_0'$.

\medskip
Case 2:  $u,v \in [n]'$. Let $p,q \in [n]$ be such that $u=p'$ and $v=q'$. Then since~$u$ and~$v$ lie in the same $a$-block, $(p,q) \in {\rm coker}(a)$ and since $u \in X \cup Z_0'$ we deduce that $p \in Z_0$. By definition, $Z_0$ is a union of $\sim$-classes, and hence also of ${\rm coker}(a)$-classes, so it follows that $q \in Z_0$, giving $v \in Z_0' \subseteq X \cup Z_0'$.

\medskip
Case 3:  $u \in [m]$ and $v \in [n]'$. As before, we have $u \in X$ and $v=q'$ for some $q \in [n]$. Since $u$ and $v$ lie in the same $a$-block, we then have $q \in Xa \subseteq Z_0$, giving $v \in Z_0' \subseteq X \cup Z_0'$.

\medskip
Case 4:  $u \in [n]'$ and $v \in [m]$. Here we have $u=p'$ for some $p \in Z_0$, and by the definition of~$Z_0$ there exists $q \in Xa \cup bY$ such that $p \sim q$. In other words, there exists a path from $p''$ to~$q''$ in the product graph $\Pi(a, b)$. 
\bit
\item If $q \in Xa$, let $x \in X$ be an element in the same $a$-block as $q'$. Since $v$ is in the same $a$-block as $u=p'$, since there is a path in $\Pi(a,b)$ from~$p''$ to~$q''$, and since~$q'$ is in the same $a$-block as $x$, it follows that~$v$ and~$x$ are in the same $ab$-block.  Then since $x \in X$ and $X \cup Y'$ is a union of $ab$-blocks, it follows that $v \in X$.
\item On the other hand, if $q \in bY$, let $y \in Y$ be an element such that $y'$ is in the same $b$-block as $q$. Since $v$ is in the same $a$-block as $u=p'$, since there is a path in $\Pi(a,b)$ from~$p''$ to~$q''$, and since $q$ is in the same $b$-block as $y'$, it follows that~$v$ and~$y'$ are in the same $ab$-block. Then since $y' \in Y'$ and $X \cup Y'$ is a union of $ab$-blocks, it follows that $v \in X$. 
\eit
Thus in both cases $v \in X  \subseteq X \cup Z_0'$.  

\bigskip

\ref{intertwine1} $\Rightarrow$ \ref{intertwine2}.  Conversely, suppose that $Z$ is an intertwining set for $(a,b)$ with respect to $(X,Y)$, i.e.~that $X \cup Z'$ is a union of $a$-blocks, and $Z \cup Y'$ is a union of $b$-blocks. In particular,~$Z$ is a union of ${\rm coker}(a)$-classes, and also a union of ${\rm ker}(b)$-classes. We first note that~$Z$ is also a union of $\sim$-classes. 
Indeed, for $z \in Z$ and $w \sim z$, it follows by the definition of $\sim$ that there is a sequence $z = z_0 , z_1 , \ldots , z_k = w$ such that each $(z_i,z_{i+1})\in\coker(a)\cup\ker(b)$; since $Z$ is a union of $\coker(a)$-classes and of $\ker(b)$-classes, it follows that each $z_i\in Z$, and in particular $w = z_k\in Z$.

Now, since $X \cup Z'$ is a union of $a$-blocks we have $Xa \subseteq Z$, and since $Z \cup Y'$ is a union of $b$-blocks, we have $bY \subseteq Z$. It then follows from the definition of $Z_0$ together with the fact that $Z$ is a union of $\sim$-classes that $Z_0 \subseteq Z$. Let $F=Z \setminus Z_0$. Since $X \cup Z'$ is a union of $a$-blocks we see that $F'$ is a union of lower $a$-blocks; dually $F$ is a union of upper $b$-blocks. Since both $Z$ and $Z_0$ are unions of $\sim$-classes, it follows that $F$ is a union of $\sim$-classes, and hence $F''$ is a union of floating components of $\Pi(a,b)$. It now remains to show that $X \cup Y'$ is a union of $ab$-blocks. 
 
Suppose then that $u \in X \cup Y'$, and that $v\in[m]\cup[t]'$ is in the same $ab$-block as $u$; we must show that $v\in X\cup Y'$.
Since $u$ and $v$ are in the same $ab$-block, there is a path from $u$ to $v$ in~$\Pi(a,b)$.

\medskip
Case 1:  If the path uses only vertices from $[m]$, then $u\in X$, $v\in[m]$, and the path uses only edges from $a$.  It follows that $u$ and~$v$ belong to the same $a$-block.  Since $X\cup Z'$ is a union of $a$-blocks, and since $u\in X$, it follows that $v\in X\cup Z'$, and hence $v\in X\sub X\cup Y'$.

\medskip
Case 2:  If the path uses only vertices from $[t]'$, then an analogous argument gives $v\in Y'\sub X\cup Y'$.

\medskip
Case 3:  Finally, suppose that the path from $u$ to $v$ involves vertices from $[n]''$.  Choosing the edge sets of the graphs representing $a$ and $b$ so that each connected component is a complete graph we can assume that the path has the form
\[
u \to p_1'' \to \cdots \to p_k'' \to v \qquad\text{for distinct $p_1,\ldots,p_k\in[n]$.}
\]
It follows that each $(p_i,p_{i+1}) \in \coker(a)\cup\ker(b)$, so that $p_1\sim p_k$.  Because of the first edge in the path, we either have $p_1 \in Xa$ (if $u \in X$) or $p_1 \in bY$ (if $u \in Y'$), from which it follows that $p_1, \ldots, p_k \in Z_0 \subseteq Z$.  Because of the last edge in the path, we see that either $v$ and~$p_k'$ lie in the same $a$-block (if $v\in [m]$) or else $v$ and $p_k$ lie in the same $b$-block (if $v\in[t]'$).  Since $X \cup Z'$ is a union of $a$-blocks, and $Z \cup Y'$ is a union of $b$-blocks, we deduce in these respective cases that $v \in X$ or $v\in Y'$, so indeed $v\in X\cup Y'$.
\end{proof}

\subsection{The main result}\label{subsect:main}

We now have all the pieces required to prove our main result:

\begin{thm}
\label{thm:partition}
Let $K$ be a non-trivial idempotent semiring. Then
\[
\varphi = \varphi_K:\P\to\M(K):a\mt\ol a
\]
is a faithful involutive tensor category representation. In particular, for each $n \in \NN$ the map~$\varphi$ gives a faithful involutive representation of the monoid $\P_n$ by $2^n\times 2^n$ matrices over~$K$.
\end{thm}

\begin{proof}
By Proposition \ref{prop:vp}, it remains only to show that $\ol{ab} = \ol a\ol b$ for any composable pair of partitions $a\in\P_{m,n}$ and $b\in\P_{n,t}$.  
By Lemmas \ref{lem:prod} and \ref{lem:intertwine}, together with the fact that $1+1=1$ in~$K$, for all $X\sub[m]$ and $Y \subseteq [t]$ we have
\begin{align*}
(\overline{a}\overline{b})_{X,Y} = \sum_{Z \subseteq [n]} \overline{a}_{X,Z}\overline{b}_{Z,Y} &= \begin{cases}
1 & \mbox{ if }X \cup Y' \mbox{ is a union of } ab\mbox{-blocks,}\\
0 & \mbox{ else}
\end{cases}\\
&= (\overline{ab})_{X,Y}.  
\end{align*}
The final statement is clear, as $\varphi$ maps $\P_{m,n}$ into $\M_{2^m,2^n}$ for any $m,n\in\NN$.
\end{proof}

\begin{rem}\label{rem:MS1997}
Consider the case in which $K$ is the Boolean semiring $\BB$.  The image under $\varphi=\varphi_\BB$ of a partition monoid $\P_n$ in $\M_{2^n}(\BB)$ has the property that for any $X, Y \subseteq [n]$ there exists $a\in\P_n$ such that $\overline{a}_{X,Y} =1$.  Recalling that $\M_{2^n}(\mathbb{B})$ is isomorphic to the monoid of binary relations $\R_{2^n}$, this means precisely that the image of $\P_n$ in $\R_{2^n}$ is a \emph{transitive} semigroup of binary relations.  This provides a concrete instance of a deep theorem of McKenzie and Schein \cite{MS1997}, which says that \emph{every} semigroup is isomorphic to a transitive semigroup of binary relations.

On the other hand, the image under $\varphi$ of a Brauer monoid $\B_n$ is \emph{not} a transitive submonoid of $\R_{2^n}$.  Rather, there are two orbits under the action of~$\B_n$: the subsets of~$[n]$ of even size, and those of odd size.  We will return to this observation in Section~\ref{sect:Bn}, and use it to construct lower-dimensional representations of $\B$.
\end{rem}

\begin{rem}\label{rem:conv}
As noted in the above proof, Theorem \ref{thm:partition} boils down to the assertion that when $K$ is an idempotent semiring, the mapping $\varphi$ is a morphism.  Conversely, we show that if the map $\P_n \rightarrow \M_n(K):a\mt\ol a$ is a morphism for some fixed $n\geq1$ (which of course holds if $\varphi$ itself is a morphism), then $K$ is idempotent.  Indeed, let $n\geq1$, and let $a = \begin{partn}{1} \bn \ \\\hline \bn \ \end{partn} \in \P_n$ be the partition with two blocks,~$\bn$~and~$\bn'$.  If we order the subsets of $\bn$ as $\es,\bn,\ldots$ (with any ordering on the remaining subsets), then
\[
\ol a = \begin{pmatrix}
1&1&0&\cdots&0\\
1&1&0&\cdots&0\\
0&0&0&\cdots&0\\
\vdots&\vdots&\vdots&\ddots&\vdots\\
0&0&0&\cdots&0
\end{pmatrix}.
\]
Since $a$ is an idempotent, $\varphi$ being a morphism implies that $\ol a$ is also an idempotent, and it is easy to see that this is equivalent to having $1+1=1$.
\end{rem}

\subsection{Minimality of the representation}\label{subsect:min}

As we noted in Section \ref{subsect:rep}, one can measure the `efficiency' of a matrix representation of a category by looking at the dimensions of the matrices in the image.  For our representation $\varphi:\P\to\M$, these dimensions are powers of $2$; specifically, we have $n \varphi=2^n$ for each $n\in\NN$.  It turns out that these dimensions are the minimum possible among all faithful \emph{involutive tensor category} representations $\P\to\M$.  Indeed, this is true for any of the (twisted) diagram categories we consider:

\newpage

\begin{thm}\label{thm:min}
Let $K$ be a unital semiring, and let $\C$ be any of $\P$, $\B$ or $\TL$, or any of their \emph{(}$d$-\emph{)}twisted counterparts.  Then for any faithful involutive tensor category representation $\psi:\C\to\M(K)$, there exists an integer $d\geq2$ such that $n\psi = d^n$ for all~$n\in\NN$.  
\end{thm}

\begin{proof}
Since $\psi$ is a tensor category representation, it acts on objects as a monoid morphism $(\NN,+)\to(\NN,\cdot)$.  It follows that with $d = 1\psi$, we have $n\psi = d^n$ for all $n\in\NN$, and it remains to check that $d\geq2$.  Now, $\psi$ maps each non-empty hom-set $\C_{m,n}$ into~$\M_{d^m,d^n}$.  Thus, we cannot have $d=0$, or else the image of $\psi$ would consist of only the $0\times0$ empty matrix, contradicting the faithfulness of $\psi$.  We also cannot have $d=1$, or else every matrix in the image of $\psi$ would be equal to its own transpose, and so then we would have $a\psi = (a\psi)^T = a^*\psi$ for all $a\in\C$, which would again contradict faithfulness.
\end{proof}

It turns out, however, that one can find smaller-dimensional (involutive) representations of some of our (twisted) diagram categories if one does not insist that they are \emph{tensor} representations.  This is the theme of Sections \ref{sect:Bn} and \ref{sect:TLn}, which resprectively treat the (twisted) Brauer and Temperley--Lieb categories.  But before this, we show in the next section that there exist faithful involutive tensor category representations of our ($d$-)twisted diagram categories of the minimal dimensions stated in Theorem \ref{thm:min}. See Section~\ref{sect:conclusion} below for some open problems regarding dimensions of representations.

\section{Twisted partition categories}\label{sect:twisted}

The central idea of the previous section can be adapted to give faithful representations of the \emph{twisted} and \emph{$d$-twisted} partition categories, by changing the semiring we work over and introducing appropriate scalar factors.

\subsection{Representations of twisted partition categories}\label{subsect:twistedP}

Here is the twisted analogue of Theorem \ref{thm:partition}:

\begin{thm}
\label{thm:twistpartition}
Let $K$ be a semiring of characteristic $0$. Then
\[
\rho = \rho_K:\P^\Phi \rightarrow \M(K):(i,a) \mapsto 2^i \bar{a}
\]
is a faithful involutive tensor category representation. In particular, for each $n \in \mathbb{N}$ the map~$\rho$ gives a faithful involutive representation of the monoid $\P_n^\Phi$ by $2^n \times 2^n$ matrices over~$K$.
\end{thm}

\begin{proof}
Consider a composable pair $(i,a),(j,b) \in \P^\Phi$, say with $a\in\P_{m,n}$ and $b\in\P_{n,t}$, and let $X\sub[m]$ and $Y \subseteq [t]$. Then by Lemmas \ref{lem:prod} and \ref{lem:intertwine} we have
\begin{align*}
((i,a)\rho \, (j,b)\rho)_{X,Y} &= (2^i\overline{a} \, 2^j\overline{b})_{X,Y} \\
&= 2^{i+j}(\overline{a} \, \overline{b})_{X,Y} \\
&= 2^{i+j}\sum_{Z \subseteq [n]} \overline{a}_{X,Z}\overline{b}_{Z,Y}\\
&= 2^{i+j}
\begin{cases} 2^{\Phi(a,b)} & \hspace{2.2mm}\mbox{ if } X \cup Y' \mbox{ is a union of } ab\mbox{-blocks,}\\
0 & \hspace{2.2mm}\mbox{ else}\end{cases}\\
&= 2^{i+j+\Phi(a,b)}
\begin{cases} 1 &                 \mbox{ if } X \cup Y' \mbox{ is a union of } ab\mbox{-blocks,}\\
0 & \mbox{ else}\end{cases}\\
&= 2^{i+j+\Phi(a,b)} (\overline{ab})_{X,Y} = (((i,a)(j,b))\rho)_{X,Y}.
\end{align*}
Thus $\rho$ preserves composition.  It is immediate from Proposition~\ref{prop:vp} that $\rho$ maps each identity $(0,\io_n)$ of $\P^\Phi$ to the identity matrix $I_{2^n}$, so $\rho$ is indeed a category representation.  It is also easy to see that $\rho$ is involutive.  

To see that $\rho$ is a tensor representation, let $i,j\in\NN$ and $a,b\in\P$ be arbitrary.  Then by the definitions, and using Proposition \ref{prop:vp}\ref{vp4}, we have
\begin{align*}
\big((i,a)\op(j,b)\big)\rho = (i+j,a\op b)\rho = 2^{i+j}\ol{a\op b} &= 2^{i+j}(\ol a\ot\ol b) \\&= 2^i\ol a \ot 2^j\ol b = (i,a)\rho\ot(j,b)\rho.
\end{align*}

Finally, to see that $\rho$ is injective, it suffices as usual to observe that $(i,a)\in\P^\Phi$ can be recovered from $2^i\ol a$.  Indeed, since this matrix contains non-zero entries, all of which are equal to $2^i$, we can recover $i$ (as $K$ has characteristic $0$).  The matrix $\ol a$ can also be recovered by changing every non-zero entry to $1$, and we can then recover $a$ itself (as $a\mt\ol a$ is injective).
\end{proof}

\begin{rem}\label{rem:ZxPn}
As in Remark \ref{rem:Z}, the set $\mathbb{Z} \times \P$ can be viewed as a twisted partition category, with respect to the same multiplication. If $2$ has a multiplicative inverse in the characteristic-$0$ semiring $K$, notice that the argument used in the proof of the previous theorem extends to give a faithful representation of $\mathbb{Z} \times \P$ also.
\end{rem}

In the following, we write $O_{k,l} \in \M_{k,l}(K)$ for the $k\times l$ zero matrix.

\begin{thm}
\label{thm:twistdpartition}
Let $d$ be a natural number, and let $K$ be a ring of characteristic $2^{d+1}$. Then 
\[
\rho^d=\rho^d_K:\P^{\Phi,d} \rightarrow \M(K):
\begin{cases}
(i,a) \mapsto 2^i \bar{a},\\
\bz_{m,n} \mt O_{2^m,2^n}
\end{cases}
\]
is a faithful involutive tensor category representation.  In particular, for each $n \in \mathbb{N}$ the map~$\rho^d$ gives a faithful involutive representation of the monoid $\P_n^{\Phi,d}$ by $2^n \times 2^n$ matrices over~$K$.
\end{thm}

\begin{proof}
Let $i,j\in\{0,1,\ldots,d\}$, and $a\in\P_{m,n}$ and $b\in\P_{n,t}$.  Then for any $X\sub[m]$ and $Y\sub[t]$, we argue exactly as in the proof of Theorem~\ref{thm:twistpartition} to obtain:
\begin{align*}
((i,a)\rho^d \, (j,b)\rho^d)_{X,Y} 
&= \begin{cases}
2^{i+j+\Phi(a,b)} (\overline{ab})_{X,Y} & \mbox{ if } i+j+\Phi(a,b) \leq d,\\
0 & \mbox{ else}
\end{cases}\\
&= (((i,a)(j,b))\rho^d)_{X,Y}.
\end{align*}
It is also clear that $\rho^d$ respects identity elements $(0,\io_n)$, and products involving zero elements~$\bz_{m,n}$ of $\P^{\Phi,d}$.
Thus $\rho^d$ is a morphism.  It is straightforward to see that $\rho^d$ respects the involution and tensor operations.  Injectivity is established exactly as in the proof of Theorem~\ref{thm:twistpartition}, keeping in mind that the elements $1,2,2^2,\ldots,2^d$ are all distinct in $K$.
\end{proof}

\subsection{A module decomposition over the twisted partition category}

For the rest of this section we fix a \emph{field} $K$ of characteristic~$0$. To study decompositions of our category representations, we need to consider the appropriate linear structure on which our representation induces an action, analogous to the way in which a monoid representation over a field induces an action on a vector space. It transpires that the correct structure is a category whose only morphisms are endomorphisms, and the local submonoids of which are vector spaces in which the composition operation is vector addition. More precisely, let $\V$ be the category whose object set is $\NN$, whose only morphisms are endomorphisms, and where for each $n\in\NN$ the endomorphism set $\V_n = \V_{n,n}$ is the vector space $K^{2^n}$ (regarded as column vectors indexed by subsets of $[n]$), with composition being vector addition. By Theorem~\ref{thm:twistpartition}, there is an obvious sense in which $\P^\Phi$ `acts' (partially) on $\V$ by:
\[
(i,a) \cdot v = 2^i\overline{a}v \qquad\text{for $i\in\NN$, $a\in\P$ and $v\in\V_{\br(a)}$.}
\]
This action gives $\V$ the structure of a $\P^\Phi$-\textit{module}, sometimes called an \textit{actegory}. See \cite{CG2024} for background theory of these structures; we note that they behave generally analogously to modules over a monoid algebra, and in particular that they admit an obvious notion of a \textit{submodule}. 
Note that $\P_{m,n}^\Phi\cdot\V_n \sub \V_m$ for all $m,n\in\NN$.

For $n\in\NN$ we write $\{v_{n,X}: X \subseteq [n]\}$ for the standard basis of $\V_n$, and define two subspaces of $\V_n$:
\[
\V_n^+: = {\rm span}_K\{v_{n,X} + v_{n,X^c}: X \subseteq [n]\} 
\ANd
\V_n^-: = {\rm span}_K\{v_{n,X} - v_{n,X^c}: X \subseteq [n]\},
\]
where here for $X \subseteq [n]$ we write $X^c = [n]\setminus X$ for the complement of $X$ in $[n]$.  
It is easy to see that $\V_n=\V_n^+\op\V_n^-$ as vector spaces.
Observe that $\V_0^+ = \V_0$, and $\V_0^- = \{0\}$, while for $n\geq1$ the spaces $\V_n^+$ and $\V_n^-$ both have dimension~$2^{n-1}$.
These collections of subspaces give rise to two induced subcategories~$\V^+$ and~$\V^-$ of $\V$.

In the next result, we write $\la v\ra$ for the $\P^\Phi$-submodule of $\V$ generated by a vector~${v\in\V}$.

\begin{thm}
\label{thm:decomp}
Let $K$ be a field of characteristic $0$. Then:
\ben
    \item \label{dc1} $\V^+$ and $\V^-$ are $\P^\Phi$-submodules of $\V$,
    \item \label{dc2} $\V^+ = \la v_{m,[m]}+v_{m,\es}\ra$ for any $m\geq0$,
    \item \label{dc3} $\V^- = \la v_{m,[m]}-v_{m,\es}\ra$ for any $m\geq1$,
    \item \label{dc4} if $\U$ is a non-zero $\P^\Phi$-submodule of $\V$, then $\U$ contains either $\V^+$ or $\V^-$.
\een
Consequently, $\V^+$ and $\V^-$ are irreducible $\P^\Phi$-modules and $\V = \V^+ \oplus \V^-$ as $\P^\Phi$-modules.
\end{thm}

\begin{proof}
\ref{dc1} Let $\al = (i,a) \in \P_{m,n}^\Phi$.  We must show that
\begin{equation}\label{eq:adot}
\al\cdot(v_{n,Y}+v_{n,Y^c}) \in \V_m^+ \AND \al\cdot(v_{n,Y}-v_{n,Y^c}) \in \V_m^- \qquad\text{for all $Y\sub[n]$.}
\end{equation}
To simplify notation in what follows, we use the abbreviation $v_Y = v_{n,Y}$ for $Y\sub[n]$.  Also, for $w\in \V_m$ and $X\subseteq[m]$, we denote by $[w]_X$ the coefficient of $v_{m,X}$ when $w$ is written as a linear combination of the basis elements of $\V_m$.  Thus, we can prove \eqref{eq:adot} by showing that
\begin{align}
\nonumber [\alpha \cdot (v_Y + v_{Y^c})]_X = [\alpha \cdot (v_Y + v_{Y^c})]_{X^c} \ANd &[\alpha \cdot (v_Y - v_{Y^c})]_X = -[\alpha \cdot (v_Y - v_{Y^c})]_{X^c}\\
\label{eq:aldot} &\qquad\text{for all $X\sub[m]$ and $Y\sub[n]$.}
\end{align}
(Note that here $Y^c=[n]\sm Y$ and $X^c=[m]\sm X$.)  To prove \eqref{eq:aldot}, fix $X\sub[m]$ and $Y\sub[n]$.  For the calculations to follow, observe that $\ol av_Y$ is column $Y$ of $\ol a$, and so $[\ol av_Y]_X = \ol a_{X,Y}$.

Since $X \cup Y'$ is a union of $a$-blocks if and only if $X^c\cup (Y^c)'$ is a union of $a$-blocks, we have
\begin{align*}
[\alpha \cdot{}& (v_Y + v_{Y^c})]_X = 2^i(\overline{a}_{X,Y} + \overline{a}_{X, Y^c})\\[1mm]
&=\begin{cases}
2^{i+1} & \mbox{ if } X \cup Z' \mbox{ is a union of }a \mbox {-blocks for both } Z=Y \mbox{ and } Z=Y^c,\\
2^i & \mbox{ if } X \cup Z' \mbox{ is a union of }a \mbox {-blocks for exactly one of } Z=Y \mbox{ or } Z=Y^c,\\
0 & \mbox{ else }
\end{cases}\\[1mm]
&=\begin{cases}
2^{i+1} & \mbox{ if } X^c \cup (Z^c)' \mbox{ is a union of }a \mbox {-blocks for both } Z=Y \mbox{ and } Z=Y^c,\\
2^i & \mbox{ if } X^c \cup (Z^c)' \mbox{ is a union of }a \mbox {-blocks for exactly one of } Z=Y \mbox{ or } Z=Y^c,\\
0 & \mbox{ else }
\end{cases}\\[1mm]
&= 2^i(\overline{a}_{X^c,Y} + \overline{a}_{X^c, Y^c}) = [\alpha \cdot (v_Y + v_{Y^c})]_{X^c},
\intertext{establishing the first equality in \eqref{eq:aldot}.  For the second, we have}
[\alpha \cdot{}& (v_Y - v_{Y^c})]_X = 2^i(\overline{a}_{X,Y} - \overline{a}_{X, Y^c})\\[1mm]
&=\begin{cases}
\mathrlap{2^i}{\phantom{2^{i+1}}}  & \mbox{ if } X \cup Z' \mbox{ is a union of }a \mbox {-blocks for } Z=Y \mbox{ but not } Z=Y^c,\\
-2^i & \mbox{ if } X \cup Z' \mbox{ is a union of }a \mbox {-blocks for } Z=Y^c \mbox{ but not } Z=Y,\\
0 & \mbox{ else }
\end{cases}\\[1mm]
&=\begin{cases}
\mathrlap{2^i}{\phantom{2^{i+1}}}& \mbox{ if } X^c \cup (Z^c)' \mbox{ is a union of }a \mbox {-blocks for } Z=Y \mbox{ but not } Z=Y^c,\\
-2^i & \mbox{ if } X^c \cup (Z^c)' \mbox{ is a union of }a \mbox {-blocks for } Z=Y^c \mbox{ but not } Z=Y,\\
0 & \mbox{ else }
\end{cases}\\[1mm]
&= 2^i(\overline{a}_{X^c,Y^c} - \overline{a}_{X^c, Y}) = -[\alpha \cdot (v_Y - v_{Y^c})]_{X^c}.
\end{align*}

\bigskip

\ref{dc2}  Fix $m\geq0$, and write $w = v_{m,[m]} + v_{m,\es}$, noting that when $m=0$ we have $w = 2v_{0,\es}$.  Since $w \in \V^+$, we just need to show that
\begin{equation}
\label{eq:dc21} v_{n,X}+v_{n,X^c} \in \la w \ra \qquad\text{for all $n\geq0$ and $X\sub[n]$.}\\
\end{equation}
To do so, fix $n\geq0$ and $X\sub[n]$, and again write $v_Y = v_{n,Y}$ for $Y\sub[n]$.  Let ${a = \begin{partn}{1}[n]\ \\\hline [m]\ \end{partn}\in\P_{n,m}}$.  (Here we slightly abuse notation; for example, if $m=0<n$, we interpret $a = \begin{partn}{1}[n]\ \\\hline \ \end{partn}\in\P_{n,0}$.)  Considering separate cases according to whether or not $m$ and/or $n$ is $0$, one can check that
\[
(0,a)\cdot w = \ol aw = \begin{cases}
2(v_{[n]}+v_\es) &\text{if $n\geq1$,}\\
2v_\es=v_{[n]}+v_\es &\text{if $n=0$.}
\end{cases}
\]
In all cases, it follows that
\begin{equation}\label{eq:dc22}
v_{[n]}+v_\es \in \la w\ra.
\end{equation}
This establishes \eqref{eq:dc21} in the case that $X=[n]$ or $X=\es$.  Now suppose $\es\subset X\subset[n]$, and let $b = \begin{partn}{2}
X & X^c \\ \hhline{-|-}
\multicolumn{2}{c}{ [m] }
\end{partn}\in\P_{n,m}$ (with similar abuse of notation when $m=0$).  We then have
\[
(0,b)\cdot w = \ol bw = 2(v_{[n]} + v_\es + v_X + v_{X^c}),
\]
which gives $(v_{[n]} + v_\es) + (v_X + v_{X^c}) \in \la w\ra$.  Combined with \eqref{eq:dc22}, it follows that ${v_X+v_{X^c} \in \la w\ra}$, completing the proof of \eqref{eq:dc21}.

\bigskip

\ref{dc3}  Fix $m\geq1$, and this time write $w = v_{m,[m]} - v_{m,\es}$.  Since $w \in \V^-$, we just need to show that
\begin{equation}
\label{eq:dc31} v_{n,X}-v_{n,X^c} \in \la w \ra \qquad\text{for all $n\geq1$ and $X\sub[n]$.}\\
\end{equation}
To do so, fix $n\geq1$ and $X\sub[n]$, and again write $v_Y = v_{n,Y}$ for $Y\sub[n]$.  With ${c = \begin{partn}{1}\bn\ \\ \bm\ \end{partn}\in\P_{n,m}}$, we have
\[
v_{[n]}-v_\es = \ol cw = (0,c)\cdot w \in \la w\ra,
\]
which gives \eqref{eq:dc31} in the case that $X=[n]$ or $X=\es$.  Now suppose $\es\subset X\subset[n]$, and let ${d = \begin{partn}{2}
X & X^c \\ \hhline{~|-}
[m] & \ 
\end{partn}\in\P_{n,m}}$.  We then have
\[
(v_{[n]} - v_\es) + (v_X - v_{X^c}) = \ol dw = (0,d) \cdot w \in \la w\ra,
\]
and again it follows that $v_X - v_{X^c} \in \la w\ra$, completing the proof of \eqref{eq:dc31}.

\bigskip

\ref{dc4} Suppose that $v \in \U$ is non-zero, say with $v\in\V_n$.  If $n=0$, then $v=\lam v_{0,\es}$ for some $\lam\in K\sm\{0\}$, and so $v_{0,\es}\in\U$; noting that $v_{n,[n]}+v_{n,\es} = 2v_{0,\es}$ in this case, it follows from part \ref{dc2} that $\V^+\sub\U$.  
Thus, for the rest of the proof we assume that $n\geq1$, and we write $v = \sum_{X \subseteq [n]} \lambda_X v_X$, where each $\lam_X\in K$, and where again $v_X=v_{n,X}$ for $X\sub[n]$.
By parts \ref{dc2} and \ref{dc3}, it suffices to show that $\U$ contains at least one of $v_{[n]}+v_\es$ or $v_{[n]} - v_{\es}$.  We consider two cases.

\medskip
Case 1: First suppose there exists $X\sub[n]$ such that $\lam_X+\lam_{X^c}\not=0$.  We show that for this $X$ we have
\begin{equation}\label{eq:dc41}
(\lam_X+\lam_{X^c})(v_{[n]}+v_\es)\in\U,
\end{equation}
which gives $v_{[n]}+v_\es\in\U$, as desired.  Working towards \eqref{eq:dc41}, we first define ${a = \begin{partn}{1} [n]\ \\\hline [n]\ \end{partn}\in\P_n}$, and note that
\[
(\lam_{[n]}+\lam_\es)(v_{[n]}+v_\es) = \ol av = (0,a)\cdot v \in \U.
\]
This establishes \eqref{eq:dc41} if $X=[n]$ or $X=\es$.  So now suppose $\es\subset X\subset[n]$, and let ${b = \begin{partn}{2}
\multicolumn{2}{c}{[n]\ } \\ \hhline{-|-}
X & X^c
\end{partn}\in\P_n}$.  We then have
\[
(\lam_{[n]}+\lam_\es+\lam_X+\lam_{X^c})(v_{[n]}+v_\es) = \ol bv = (0,b)\cdot v \in \U.
\]
It then follows that $(\lam_X+\lam_{X^c})(v_{[n]}+v_\es) = (0,b)\cdot v - (0,a)\cdot v\in\U$, completing the proof in this case.

\medskip
Case 2: Now suppose $\lam_X+\lam_{X^c}=0$ for all $X\sub[n]$.  Since $v\not=0$ we can fix $X\sub[n]$ such that $\lam_X\not=0$.  We show that for this $X$ we have
\begin{equation}\label{eq:dc42}
\lam_X(v_{[n]} - v_{\es}) \in \U,
\end{equation}
which gives $v_{[n]} - v_{\es} \in \U$, as desired.
With $c = \begin{partn}{1} \bn\ \\ \bn\ \end{partn}\in\P_n$, we first note (keeping $\lam_{[n]}+\lam_\es=0$ in mind) that
\[
\lam_{[n]}(v_{[n]} - v_{\es}) = \lam_{[n]}v_{[n]} + \lam_\es v_{\es} = \ol c v = (0,c)\cdot v \in \U.
\]
This proves \eqref{eq:dc42} if $X=[n]$ or $X=\es$.  So now suppose $\es\subset X\subset[n]$, put ${d = \begin{partn}{2}
[n] & \  \\ \hhline{~|-}
X & X^c
\end{partn}\in\P_n}$, and note that
\[
(\lam_{[n]}+\lam_X)(v_{[n]} - v_{\es}) = (\lam_{[n]}+\lam_X)v_{[n]} + (\lam_\es+\lam_{X^c})v_\es = \ol d v = (0,d)\cdot v \in \U.
\]
This time we obtain $\lam_X(v_{[n]} - v_{\es}) = (0,d)\cdot v - (0,c)\cdot v \in \U$,
completing the proof of~\eqref{eq:dc42}, and hence of the result.
\end{proof}

We also have a monoid version of Theorem \ref{thm:decomp}.  It does not follow directly from the theorem itself, but it does follow from the proof, of which we only need the $m=n$ cases, where all partitions constructed belong to $\P_n$.  In the second part of the following statement, $\la v\ra$ denotes the $\P_n^\Phi$-submodule of $\V_n$ generated by $v\in\V_n$.

\begin{thm}
\label{thm:decomp_mon}
Let $K$ be a field of characteristic $0$, and let $n\geq0$. Then:
\ben
    \item \label{dcm1} $\V_n^+$ and $\V_n^-$ are $\P_n^\Phi$-submodules of $\V_n$,
    \item \label{dcm2} $\V_n^+ = \la v_{n,[n]}+v_{n,\es}\ra$ and $\V_n^- = \la v_{n,[n]}-v_{n,\es}\ra$,
    \item \label{dcm3} if $\U$ is a non-zero $\P_n^\Phi$-submodule of $\V_n$, then $\U$ contains either $\V_n^+$ or $\V_n^-$.
\een
Consequently, $\V_n^+$ and $\V_n^-$ are irreducible $\P_n^\Phi$-modules and $\V_n = \V_n^+ \oplus \V_n^-$ as $\P_n^\Phi$-modules.
\end{thm}

\begin{rem}\label{rem:decomp}
Neither $\V_n^+$ nor $\V_n^-$ is a faithful $\P_n^\Phi$-module for $n\geq2$.  Consequently, neither $\V^+$ nor $\V^-$ is a faithful $\P^\Phi$-module.  To verify this, fix $\emptyset \subset Y \subset [n]$, and define the partitions
\[
a = \begin{partn}{2}
Y & Y^c \\ \hhline{~|-}
Y & Y^c
\end{partn}\COMMA
b = \begin{partn}{2}
Y & Y^c \\ \hhline{~|-}
Y^c & Y
\end{partn}\COMMA
c = \begin{partn}{2}
Y & Y^c \\ \hhline{-|-}
Y & Y^c
\end{partn}\AND
d = \begin{partn}{2}
Y & Y^c \\ \hhline{-|-}
\multicolumn{2}{c}{\bn}
\end{partn},
\]
all from $\P_n$.  Also let $\CC = \{\emptyset, Y, Y^c, [n]\}$, noting that these sets index the non-zero rows of $\ol a$, $\ol b$, $\ol c$ and $\ol d$.  The same sets also index the non-zero columns of $\ol a$, $\ol b$ and $\ol c$, whereas the non-zero columns of $\ol d$ are indexed by $\es$ and $[n]$.  Again writing $v_X=v_{n,X}$ for $X\sub[n]$, it follows that
\[
\ol av_X = \ol bv_X = \ol cv_X = \ol dv_X = 0 \qquad\text{for all $X \not \in \CC$.}
\]
This then gives
\begin{equation}\label{eq:XnotinC}
(i,f)\cdot(v_X+v_{X^c}) = 0 = (i,f)\cdot(v_X-v_{X^c})  \qquad\text{for all $X \not \in \CC$, and for $f\in\{a,b,c,d\}$.}
\end{equation}
Next, for $X \in \CC$ it is straightforward to verify that 
\begin{align*}
\ol av_X &= \begin{cases} 
v_\emptyset + v_{Y^c} & \mbox{ if } X=\emptyset \mbox{ or } X = Y^c,\\
v_Y + v_{[n]} & \mbox{ if } X=Y \mbox{ or } X = [n],\end{cases}
\intertext{and}
\ol bv_X &= \begin{cases}
v_\emptyset + v_{Y^c} & \mbox{ if } X=\emptyset \mbox{ or } X = Y,\\
v_{Y} + v_{[n]} & \mbox{ if } X=Y^c \mbox{ or } X = [n].\end{cases}
\intertext{Hence $(i,a) \cdot (v_X + v_{X^c}) = 2^i(v_\emptyset + v_{Y} + v_{Y^c} + v_{[n]})  = (i,b)\cdot (v_X + v_{X^c})$ for all $X \in \CC$.  Combined with \eqref{eq:XnotinC}, this demonstrates that $\V_n^+$ is not a faithful $\P_n^\Phi$-module. Meanwhile,}
\ol cv_X &= v_{\emptyset} + v_Y + v_{Y^c} + v_{[n]} \hspace{7mm}\text{for all $X\in\CC$,}
\intertext{and}
\ol dv_X &= \begin{cases}
v_{\emptyset} + v_Y + v_{Y^c} + v_{[n]} &\text{if $X=\es$ or $X=[n]$,}\\
0 &\text{if $X=Y$ or $X=Y^c$.}
\end{cases}
\end{align*}
Hence $(i,c) \cdot (v_X - v_{X^c}) = 0  = (i,d)\cdot (v_X - v_{X^c})$ for all $X \in \CC$, and it again follows that~$\V_n^-$ is not a faithful $\P_n^\Phi$-module.
\end{rem}

\section{The Brauer category}\label{sect:Bn}

As before, fix a non-trivial idempotent semiring $K$.  In this section we consider the Brauer category $\B$, and show that the restriction to~$\B$ of the representation ${\varphi:\P\to\M}$ induces a pair of `reduced' representations of~$\B$ (and hence also of the Temperley--Lieb category~$\TL$).  Whereas $\varphi$ maps a hom-set $\P_{m,n}$ into $\M_{2^m,2^n}$, the reduced representations map $\B_{m,n}$ into $\M_{2^{m-1},2^{n-1}}$ when $m,n\geq1$ (with small adjustments to these dimensions when~$m=0$ or~$n=0$).
These representations are faithful when restricted to the odd-degree subcategory $\B_\odd = \bigcup_{m,n\in\NN}\B_{2m+1,2n+1}$.  
We also obtain corollaries for the corresponding twisted categories.

Throughout this section, for each $n\in\NN$ we order the subsets of $[n]$ so that the subsets of odd cardinality are earlier in the ordering than those of even cardinality.  Consider now a Brauer partition $a\in\B$, and its image $\ol a\in\M$ under $\varphi$.  Since each $a$-block has size $2$, it is immediate from Definition \ref{def:therep} that $\overline{a}_{X,Y} = 1$ implies that $|X| \equiv |Y| \bmod 2$. Thus we have
\[
\overline{a} =  \begin{pmatrix} \overline{a}^1 & O \\
   O & \overline{a}^0\\
\end{pmatrix},
\]
where:
\bit
\item $\overline{a}^1$ is the restriction of $\overline{a}$ to the rows and columns indexed by subsets of odd cardinality,
\item $\overline{a}^0$ is the restriction of $\overline{a}$ to the rows and columns indexed by subsets of even cardinality, and
\item $O$ denotes zero matrices of appropriate dimensions.
\eit
Note that $\ol a^1$ is an empty matrix if $0\in\{m,n\}$, as there are no odd-sized subsets of~${[0]=\es}$; in this case $\ol a=\begin{pmatrix}O&\ol a^0\end{pmatrix}$ or $\ol a=\begin{pmatrix}O\\\ol a^0\end{pmatrix}$, as appropriate.  Otherwise,~$\ol a^1$ and~$\ol a^0$ both belong to $\M_{2^{m-1},2^{n-1}}$.  

\begin{rem}\label{rem:notens}
It follows immediately from the block structure above, together with Theorem \ref{thm:partition}, that
\[
\B\to\M:a \mapsto \overline{a}^1 \AND \B\to\M:a \mapsto \overline{a}^0
\]
are both involutive category representations.  These are not \emph{tensor} representations, however, as can be seen by considering the dimensions of the matrices.  For example, if $a\in\B_{m,n}$ and $b\in\B_{s,t}$, where $m,n,s,t\geq1$, then for $i=1,2$ we have
\[
\ol{a\op b}^i \in \M_{2^{m+s-1},2^{n+t-1}}, \qquad\text{but}\qquad \ol a^i \ot \ol b^i \in \M_{2^{m+s-2},2^{n+t-2}}.
\]
\end{rem}

It turns out that the above `reduced' representations of $\B$ are injective when restricted to the odd-degree subcategory $\B_\odd = \bigcup_{m,n\in\NN}\B_{2m+1,2n+1}$:

\begin{thm}
\label{thm:Brauer_odd}
Let $K$ be a non-trivial idempotent semiring. Then $a \mapsto \overline{a}^1$ and $a \mapsto \overline{a}^0$ are both faithful involutive category representations of $\B_\odd$ in $\M(K)$.  These both yield faithful involutive representations of the monoids $\B_n$ for each odd $n$ by $2^{n-1} \times 2^{n-1}$ matrices over $K$.
\end{thm}

\begin{proof}
From the discussion above, it suffices to show that the maps $a \mapsto \overline{a}^1$ and $a \mapsto \overline{a}^0$ are injective.  To do so, let $a \in \B_{m,n}$, where $m,n\in\NN$ are both odd.  We show that $a$ can be reconstructed from either of $\ol a^1$ or~$\ol a^0$.

For each $x\in[m]$ and $y \in [n]$, we have $\overline{a}^1_{\{x\},\{y\}} = 1$ if and only if $\{x,y'\}$ is a transversal of $a$. Since all $a$-blocks have size $2$, we may therefore deduce all transversals from $\overline{a}^1$.  Moreover, since $m$ and $n$ are odd, $a$ must contain at least one transversal, say $\{x,y'\}$. Then for distinct $i,j \in [m] \setminus \{x\}$ we have $\overline{a}^1_{\{x,i,j\},\{y\}} = 1$ if and only if $\{i,j\}$ is an upper $a$-block.  Thus, we can recover all upper $a$-blocks from $\ol a^1$, and similarly for the lower blocks.

Similarly, the upper and lower blocks of $a$ can be recovered from $\ol a^0$ by looking at the entries $\ol a^0_{\{i,j\},\es}$ and $\ol a^0_{\es,\{i,j\}}$.  Knowing these, we can also recover $\dom(a)$ and $\codom(a)$, as the elements of $[m]$ or $[n]$ not belonging to upper or lower blocks, and hence we also know ${r:=\rank(a) = |{\dom(a)}| = |{\codom(a)}|}$.  If $r=1$, then ${\dom(a)=\{x\}}$ and ${\codom(a)=\{y\}}$ for some $x\in[m]$ and $y\in[n]$, and the final remaining $a$-block is the transversal $\{x,y'\}$.  When $r\geq3$ (recalling that $r$ is odd), we recover the transversals of $a$ as follows.  Fix some $x\in\dom(a)$, and let $y,z\in\dom(a)$ be such that $\{x,y,z\}$ has size~$3$.  Then there exist unique subsets $U,V\sub[n]$ of size $2$ such that $\ol a_{\{x,y\},U}^0 = \ol a_{\{x,z\},V}^0 = 1$.  The transversal of~$a$ containing $x$ is then $\{x,w'\}$, where $w$ is the unique element of~$U\cap V$.
\end{proof}

\begin{rem}\label{rem:Brauer_even}
In the case where $m$ and $n$ are even, the same argument applies to allow us to recover all the transversals of $a\in\B_{m,n}$ (if any) from $\ol a^1$, and \emph{in the case where there is at least one transversal} we can also recover the upper and lower blocks. However, the representation $a \mapsto \overline{a}^1$ is \emph{not} injective on even-degree hom-sets $\B_{m,n}$ (apart from trivially in the cases in which $m,n\leq2$), since it does not distinguish any two rank-$0$ partitions.  Indeed, $\ol a^1=O$ is an appropriately-sized zero matrix when $\rank(a)=0$. 
It follows that the kernel of the representation $\B_\even\to\M:a\mt\ol a^1$ is the Rees congruence over the rank-$0$ ideal of $\B_\even$; see \cite[Chapter 7]{ER2023} for more on congruences of Brauer categories.

The representation $\B_\even\to\M:a\mt\ol a^0$ is also not faithful, and we can again compute its kernel.  One can still recover the non-transversals of $a\in\B_\even$ from the entries $\ol a_{\{i,j\},\es}^0$ and $\ol a_{\es,\{i,j\}}^0$.  One can also recover the transversals \emph{if there are at least four of them} (recall that $\rank(a)$ is even), using the argument at the end of the above proof.  However, this argument breaks down when $\rank(a)=2$.  For example, consider
\begin{equation}\label{eq:ab}
a = \custpartn{1,2,3,4,10,11}{1,2,3,4,7,8}{\stline11\stline22\uarc34\darc34\uarc{10}{11}\darc78\udotted4{10}\ddotted47}
\AND
b = \custpartn{1,2,3,4,10,11}{1,2,3,4,7,8}{\stline12\stline21\uarc34\darc34\uarc{10}{11}\darc78\udotted4{10}\ddotted47},
\end{equation}
both from $\B_{m,n}$, where $m,n\geq2$ are even (shown here in the case $m>n$).  Then for $c=a,b$ we have  $\overline{c}^0_{X,Y} = 1$ if and only if $X$ is a union of ${\rm ker}(c)$-classes, $Y$ is a union of ${\rm coker}(c)$-classes, and $\{1,2\}$ is either a subset of both $X$ and $Y$ or else disjoint from both $X$ and $Y$. Since ${\rm ker}(a) = {\rm ker}(b)$ and ${\rm coker}(a) = {\rm coker}(b)$, this gives $\overline{a}^0=\overline{b}^0$.
It follows from all of this that the kernel of the representation $\B_\even\to\M:a\mt\ol a^0$ is the congruence
\begin{align*}
\big\{(a,b)\in\B_\even\times\B_\even:  \bd(a)=\bd(b),\ \br(a)=\br(b),\ &\rank(a)=\rank(b)=2,\\
& \ker(a)=\ker(b),\ \coker(a)=\coker(b)\big\},
\end{align*}
which was denoted $\mu_{I_0,\S_2}^{I_\omega}$ in \cite[Theorem 7.12]{ER2023}.
\end{rem}

As in the previous section, the above representations can be adapted to provide faithful representations of the ($d$-)twisted Brauer categories over a semiring of appropriate characteristic.  We only provide a brief sketch of a proof in the infinite case.

\begin{cor}
\label{cor:twistBrauer}
Let $K$ be a semiring of characteristic $0$.  Then $(i,a) \mapsto 2^i \bar{a}^1$ and ${(i,a) \mapsto 2^i \bar{a}^0}$ are both faithful involutive category representations of $\B_\odd^\Phi$ in $\M(K)$.  These both yield faithful involutive representations of the monoids $\B_n^\Phi$ for each odd $n$ by $2^{n-1} \times 2^{n-1}$ matrices over $K$.
\end{cor}

\begin{proof}
The key point is that for $a\in\B_{m,n}$ and $b\in\B_{n,t}$ we have
\[
\begin{pmatrix}
\ol a^1\ol b^1 & O\\
O & \ol a^0\ol b^0
\end{pmatrix}
=
\ol a \ol b
=
2^{\Phi(a,b)}\ol{ab}
=
\begin{pmatrix}
2^{\Phi(a,b)}\ol{ab}^1 & O\\
O & 2^{\Phi(a,b)}\ol{ab}^0
\end{pmatrix}
,
\]
from which it quickly follows that the stated maps are both morphisms.  Injectivity is established exactly as in Section \ref{subsect:twistedP}.
\end{proof}

\begin{rem}
As in Remark \ref{rem:ZxPn}, we can also obtain a version of Corollary \ref{cor:twistBrauer} with $\NN$ replaced by $\ZZ$, if we additionally assume that $2$ has a multiplicative inverse in $K$.
\end{rem}

\begin{cor}
\label{cor:twistdBrauer}
Let $d$ be a natural number, and $K$ be a ring of characteristic $2^{d+1}$.  Then $(i,a) \mapsto 2^i \bar{a}^1$ and $(i,a) \mapsto 2^i \bar{a}^0$ are both faithful involutive category representations of~$\B_\odd^{\Phi,d}$ in $\M(K)$.  \emph{(}In the above representations, zero elements $\bz_{m,n}$ map to zero matrices $O_{2^{m-1},2^{n-1}}$.\emph{)}  These both yield faithful involutive representations of the monoids~$\B_n^{\Phi,d}$ for each odd $n$ by $2^{n-1} \times 2^{n-1}$ matrices over $K$.
\end{cor}

\section{The Temperley--Lieb category}\label{sect:TLn}

We now turn our attention to the Temperley--Lieb category $\TL$, again with the intention of providing a lower-dimensional matrix representation. 
Since $\TL\sub \B$, Theorem~\ref{thm:Brauer_odd} immediately tells us that $a\mt\ol a^1$ and $a\mt\ol a^0$ are both faithful representations of $\TL_\odd$.  The representation $a\mt\ol a^1$ of $\TL_\even$ is still not faithful; as explained in Remark~\ref{rem:Brauer_even}, we have $\ol a^1=O$ for all $a\in \TL$ of rank $0$.  However, the elements $a,b\in \B_\even$ in \eqref{eq:ab} do not both belong to $\TL$, and in fact we have the following:

\begin{thm}
\label{thm:TL_even}
Let $K$ be a non-trivial idempotent semiring. Then $a \mapsto \overline{a}^0$ is a faithful involutive category representation of $\TL$ in $\M(K)$.  In particular, for each $n\geq 1$ this yields a faithful involutive representation of the monoid $\TL_n$ by $2^{n-1}\times 2^{n-1}$ matrices over $K$.
\end{thm}

\begin{proof}
As usual, it suffices to show that $a \mapsto \overline{a}^0$ is injective on $\TL$.  To do so, let ${a \in \TL}$.  As in the proof of Theorem \ref{thm:Brauer_odd}, we can recover the upper and lower blocks of~$a$ from~$\ol a^0$, and also the sets $\dom(a)$ and $\codom(a)$.  If $\dom(a)=\{x_1<\cdots<x_r\}$ and $\codom(a)={\{y_1<\cdots<y_r\}}$, then it follows by planarity that the transversals of $a$ are $\{x_1,y_1'\},\ldots,\{x_r,y_r'\}$. 
\end{proof}

\begin{rem}
As in Corollaries \ref{cor:twistBrauer} and \ref{cor:twistdBrauer}, the representation in Theorem \ref{thm:TL_even} can be adapted to give faithful representations of ($d$-)twisted Temperley--Lieb categories and monoids over (semi)rings of appropriate characteristics.
\end{rem}

The main goal in the rest of this section is to improve on Theorem \ref{thm:TL_even} by giving a faithful involutive category representation of $\TL$ utilising matrices of strictly smaller dimension.  Specifically, $a\in\TL_{m,n}$ is mapped to an $f_m\times f_n$ matrix, where these are Fibonacci numbers; see Theorem \ref{thm:TLfib}. With this goal in mind, we begin with some simple parity-related observations, which were noted in \cite[Chapter 8]{ER2023}:

\begin{lem}
\label{lem:parity}
Let $a \in \TL_{m,n}$.
\ben
    \item If $\{i,j\}$ is an upper $a$-block, then $i \not\equiv j \bmod 2$.
    \item If $\{i',j'\}$ is a lower $a$-block, then $i \not\equiv j \bmod 2$.
    \item If $\{i,j'\}$ is a transversal  $a$-block, then $i \equiv j \bmod 2$.
    \item If $r=\rank(a)$, then $r \equiv m \equiv n \bmod 2$. 
    \item If $\dom(a) = \{x_1<\cdots<x_r\}$, then $x_i\equiv i\bmod2$ for each $i$.
    \item If $\codom(a) = \{y_1<\cdots<y_r\}$, then $y_i\equiv i\bmod2$ for each $i$.
\een
\end{lem}

Our reduced representation of $\TL$ will involve mapping ${a\in\TL}$ to a sub-matrix of $\ol a$, induced on rows and columns corresponding to certain special subsets:

\begin{defn}
For a natural number $n\in\NN$, and for a subset $X\sub[n]$, we say that~$X$ is an \emph{even-gap} subset of $[n]$ if the complement $[n] \setminus X$ is a (possibly empty) disjoint union of intervals each containing an even number of points.  We write $\X_n$ for the set of all even-gap subsets of $[n]$.
\end{defn}

\begin{rem}\label{rem:eg}
It is easy to check that a non-empty subset $X \subseteq [n]$ is an even-gap subset  if and only if the least element is odd, the greatest element has the same parity as $n$ and when ordered from smallest to largest the elements of $X$ have alternating parity. In other words, $X=\{x_1<\cdots<x_r\} \subseteq [n]$ is an even-gap subset if and only if $x_i\equiv i\bmod2$ for all $i$ and $r\equiv n\bmod 2$.  In particular, it follows from Lemma~\ref{lem:parity} that  for any $a\in \TL_{m,n}$, $\dom(a)$ and $\codom(a)$ are even-gap subsets of $[m]$ and $[n]$, respectively.  Conversely, given $X\in\X_m$ and $Y\in\X_n$ with $|X|=|Y|$, it is not hard to see that there exists $a\in\TL_{m,n}$ with $\dom(a)=X$ and $\codom(a)=Y$.
\end{rem}

\begin{ex}
The even-gap subsets of $[4]$ are: $\emptyset$, $\{1,2\}$, $\{1,4\}$, $\{3,4\}$ and $\{1,2,3,4\}$.
The even-gap subsets of $[5]$ are: $\{1\}$, $\{3\}$, $\{5\}$, $\{1,2,3\}$, $\{1,2,5\}$, $\{1,4,5\}$, $\{3,4,5\}$ and $\{1,2,3,4,5\}$.
\end{ex}

We now show that the number of even-gap subsets of $[n]$ is the Fibonacci number $f_n$, where we take the convention that $f_0 = f_1 = 1$ and $f_n = f_{n-1} + f_{n-2}$ for all $n\geq 2$:

\begin{lem}
For $n\in\NN$ we have $|\X_n| = f_{n}$.
\end{lem}

\begin{proof}
It is clear that $|\X_0|=|\X_1|=1$.  For $n\geq2$, we have the disjoint union ${\X_n = \Y_n \sqcup \Z_n}$, where
\[
\Y_n = \{A\in\X_n:n\in A\} \AND \Z_n = \{A\in\X_n:n\not\in A\}.
\]
One can check that $\X_{n-1}\to\Y_n:A\mt A\cup\{n\}$ is a bijection.  By definition, any $A\in\Z_n$ also satisfies $n-1\not\in\Z_n$.  So in fact we have $\Z_n = \X_{n-2}$.  It follows that
\[
|\X_n| = |\Y_n| + |\Z_n| = |\X_{n-1}| + |\X_{n-2}|.  \qedhere
\]
\end{proof}

For the remainder of this section for a semiring $K$ we shall view elements of $\M_{f_m,f_n}(K)$ as matrices with rows and columns indexed by the elements of $\X_m$ and $\X_n$, respectively. 

\begin{defn}
\label{def:thefibrep}
Given a (non-trivial) semiring $K$, we define the mapping
\[
\mu = \mu_K: \TL \rightarrow \M(K):a\mapsto \ool a,
\]
where for each $a \in \TL_{m,n}$ the matrix $\ool a\in \M_{f_m,f_n}(K)$ is defined, for $X\in\X_m$ and $Y\in\X_n$, by
\[
\ool a_{X,Y} = \ol a_{X,Y} = \begin{cases} 1 & \mbox{ if } X \cup Y' \mbox{ is a (possibly empty) union of } a\mbox{-blocks,}\\
0 & \mbox{ else.}
\end{cases}
\]
That is, $\ool a$ is the induced sub-matrix of $\ol a$ on the rows and columns indexed by subsets from~$\X_m$ and $\X_n$.
\end{defn}

In Theorem \ref{thm:TLfib} below, we show that $\mu$ is a faithful involutive category representation of $\TL$ in $\M$ when $K$ is an idempotent semiring.  It is not a tensor representation, however, as can again be checked by considering the dimensions of the matrices $\ool{a\op b}$ and $\ool a\ot\ool b$, for $a,b\in\TL$ (as in Remark \ref{rem:notens}).

For the proof of Theorem \ref{thm:TLfib}, we need a number of technical lemmas.
In what follows, for $X \sub [n]$ and $j \in [n]$ we will write ${X_{<j} = \set{x\in X}{x<j}}$, with analogous meanings for $X_{\leq j}$, $X_{>j}$ and so on.  First we have a simple property of even-gap subsets:

\begin{lem}\label{lem:opp}
For any $X\in\X_n$, and for any $x\in X$, we have
\[
x\equiv|X_{\leq x}| \bmod2 \AND x\not\equiv|X_{< x}| \bmod2.
\]
\end{lem}

\begin{proof}
Writing $X=\{x_1<\cdots<x_r\}$, this follows quickly from the characterisation of even-gap subsets in Remark \ref{rem:eg}.
\end{proof}

We also have the following method for constructing an even-gap subset of $[n+1]$ from an even-gap subset of $[n]$:

\newpage

\begin{lem}\label{lem:fn0}
Let $Y\in\X_n$ and $1\leq i\leq n$.  Then with $Y_1 = Y_{<i}$ and $Y_2 = Y_{\geq i}$, we have
\ben
\item $Y_1 \cup \{i\} \cup (Y_2+1) \in \X_{n+1} \iff i \not\equiv |Y_1| \bmod2$,
\item $Y_1 \cup \{i,i+1\} \cup (Y_2+2) \in \X_{n+2} \iff i \not\equiv |Y_1| \bmod2$.
\een
\end{lem}

\begin{proof}
Write $Y_1 = \{x_1<\cdots<x_k\}$ and ${Y_2 = \{x_{k+1}<\cdots<x_l\}}$, noting that 
\[
Y_1 \cup \{i\} \cup (Y_2+1) = \{x_1<\cdots<x_k < i < x_{k+1}+1<\cdots<x_l+1\}.
\]
The first claim is then easily checked, considering cases according to whether one or both of $k=|Y_1|$ or $l-k=|Y_2|$ is~$0$. The second
claim follows from two applications of the first.
\end{proof}

To show that $\mu:\TL\to\M$ is a representation (when $K$ is an idempotent semiring), we need to show that $\ool{ab} = \ool a\, \ool b$ for any composable pair of Temperley--Lieb partitions $a,b\in\TL$.
The next two lemmas do this for two very specific choices of $b$, and for arbitrary $a$, and their proofs refer to the sets $Z_0$ from Definition \ref{defn:Z0}.

\begin{lem}\label{lem:fn1}
Let $a\in\TL_{m,n}$, and let $b = \custpartn{1,3,4,6}{1,3,4,5,6,8}{\stline11\stline33\stline46\stline68\darc45\udotted13\ddotted13 \udotted46\ddotted68 \vertlab11 \vertlab4i\vertlab6n}\in\TL_{n,n+2}$, where ${1\leq i\leq n+1}$.  Then $\ool{ab} = \ool a\, \ool b$.
\end{lem}

\begin{proof}
Let $X\in\X_m$ and $Y\in\X_{n+2}$.  We must show that $\ool{ab}_{X,Y} = 1 \iff ( \ool a\, \ool b)_{X,Y} = 1$.

Suppose first that $( \ool a\, \ool b)_{X,Y} = 1$.  Then by definition of matrix multiplication over $K$, there exists $Z \in \X_n$ such that $\ool{a}_{X,Z}=\ool{b}_{Z,Y} = 1$. It follows that $\ol{a}_{X,Z}=\ol{b}_{Z,Y} = 1$, and hence $\ol{ab}_{X,Y} = (\ol a\; \ol{b})_{X,Y} = 1$. Since $X, Y \in \X_n$, this in turn yields $\ool{ab}_{X,Y}= \ol{ab}_{X,Y}= 1$.

Conversely, suppose $\ool{ab}_{X,Y} = 1$.  Then $\ol{ab}_{X,Y} = 1$, so it follows from Lemma \ref{lem:intertwine} that~$Z_0$ is an intertwining set for $(a,b)$ with respect to $(X,Y)$.  Thus, $\ol a_{X,Z_0} = 1 = \ol b_{Z_0,Y}$.  The proof will therefore be complete if we can show that $Z_0\in\X_n$, as we will then have $\ool a_{X,Z_0} = 1 = \ool b_{Z_0,Y}$, and hence $( \ool a\, \ool b)_{X,Y} = 1$.  Since $\ol b_{Z_0,Y} = 1$, and since $b$ has no upper blocks, it follows from Lemma \ref{lem:ones} that $Y$ is a union of $\coker(b)$-classes, and that $Z_0 = bY$.  Writing $Y_1 = Y_{<i}$ and $Y_2 = Y_{\geq i+2}$, we therefore have
\[
Y = Y_1\cup Y_2 \qquad\text{or}\qquad Y = Y_1 \cup\{i,i+1\}\cup Y_2 .
\]
Since $Y\in\X_{n+2}$, it follows in either case that $Y_1\cup Y_2 \in \X_{n+2}$.  (For the second case it is clear that deleting two consecutive elements of an even-gap subset leaves another even-gap subset.)  In either case we also have $bY = b(Y_1\cup Y_2)$, since $i,i+1\not\in\codom(b)$, and so
\[
Z_0 = bY = b(Y_1\cup Y_2) = Y_1 \cup (Y_2-2).
\]
From Remark \ref{rem:eg}, it is also clear that $Y_1\cup Y_2\in\X_{n+2}$ implies $Y_1 \cup (Y_2-2)\in\X_n$.  Thus we indeed have~${Z_0\in\X_n}$, and as noted above this completes the proof.
\end{proof}

In the next proof, given $x,i,j\in\NN$ with $1\leq i<j$, we say $x$ is \emph{nested} by $\{i,j\}$ if $i<x<j$.  We say a subset of $\NN$ is nested by $\{i,j\}$ if each element of the subset is nested.  We similarly speak of nesting among elements of $\NN'$.

\begin{lem}\label{lem:fn2}
Let $a\in\TL_{m,n}$ where $n\geq2$, and let $b = \custpartn{1,3,4,5,6,8}{1,3,4,6}{\stline11\stline33\stline64\stline86\uarc45\udotted13\ddotted13 \ddotted46\udotted68 \vertlab11 \vertlab4i\vertlab8n}\in\TL_{n,n-2}$, where ${1\leq i\leq n-1}$.  Then $\ool{ab} = \ool a\, \ool b$.
\end{lem}

\begin{proof}
This time we must show that $\ool{ab}_{X,Y} = 1 \iff ( \ool a\, \ool b)_{X,Y} = 1$ for all $X\in\X_m$ and $Y\in\X_{n-2}$.  The backward implication is treated in the same way as for Lemma~\ref{lem:fn1}, while the forward implication again boils down to showing that $\ool{ab}_{X,Y} = 1$ implies $Z_0\in\X_n$.  

So for the rest of the proof we assume that $\ool{ab}_{X,Y} = 1$, i.e.~that $X \cup Y'$ is a union of $ab$-blocks.  As before, it follows that $\overline{a}_{X,Z_0} = \overline{b}_{Z_0,Y} = 1$.  Since $b$ has no lower blocks, it follows from Lemma \ref{lem:ones} (and $\overline{b}_{Z_0,Y} = 1$) that $Z_0$ is a union of ${\rm ker}(b)$-classes, and that $Y = Z_0b$.  Writing $Y_1 = Y_{<i}$ and $Y_2 = Y_{\geq i}$, so that $Y = Y_1 \cup Y_2$, we therefore have
\[
Z_0 = Y_1\cup(Y_2+2) \qquad\text{or}\qquad Z_0 = Y_1\cup\{i,i+1\}\cup(Y_2+2).
\]
Since $Y\in\X_{n-2}$, it follows in either case (via Remark \ref{rem:eg}) that $Y_1\cup(Y_2+2)\in\X_n$.  Thus, we are left to show that $Z_0$ belongs to $\X_n$ in the second case.  So we assume that $Z_0 = Y_1\cup\{i,i+1\}\cup(Y_2+2)$ for the rest of the proof, and we must show that $Z_0\in\X_n$.  By Lemma~\ref{lem:fn0}(ii), it is enough to show that
\begin{equation}\label{eq:iY1}
i \not\equiv |Y_1| \bmod2.
\end{equation}
We now consider three cases, according to the form of the $a$-block containing $i'$.  It may help to consult Figure \ref{fig:2}, which illustrates the product graph $\Pi(a,b)$ in each case.

\medskip
Case 1: Suppose first that $a$ contains a transversal $\{x,i'\}$ for some $x$, and note then that $x\in X$, as $i\in Z_0$ and $X\cup Z_0'$ is a union of $a$-blocks (the latter because $\ol a_{X,Z_0}=1$); see Figure \ref{fig:2a}. By Lemma~\ref{lem:parity} we have $x\equiv i \bmod2$, and so to verify  \eqref{eq:iY1} in this case it suffices to show that ${x\not\equiv |Y_1| \bmod2}$. Since $X \cup Z_0'$ is a union of $a$-blocks, and since $\{x,i'\}$ is an $a$-block, it follows from planarity that $X_{<x} \cup (Z_0)_{<i}' = X_{<x} \cup Y_1'$ is a union of $a$-blocks.  Since all $a$-blocks have size $2$, it follows that $|X_{<x}| \equiv |Y_1| \bmod2$.  Since $X\in\X_m$ and $x\in X$, it now follows from Lemma~\ref{lem:opp} that $|Y_1|=|X_{<x}| \not\equiv x \bmod2$, as required.

\medskip

Case 2: Next suppose that $a$ contains a lower block $\{x',i'\}$ where $x<i$; see Figure~\ref{fig:2b}.  By Lemma~\ref{lem:parity} we have $x\not\equiv i \bmod2$, and so to verify  \eqref{eq:iY1} in this case it suffices to show that $x\equiv |Y_1| \bmod2$. Since~$i\in Z_0$, and since $X\cup Z_0'$ is a union of $a$-blocks, it follows that $x\in Z_0$.  Thus, $x\in Y_1$ (as~$x<i$).  
Define the (possibly empty) set
\[
U = \set{u\in Y_1}{x<u<i} = \{u_1<\cdots<u_k\} \subseteq Z_0,
\]
noting that $\max(Y_1)=u_k$ if $k\geq1$, and $\max(Y_1)=x$ otherwise.
Since $\{x',i'\}$ is an $a$-block, it follows by planarity that every element of $U'$ is contained in a lower $a$-block that is nested by $\{x',i'\}$.  Since $Z_0$ is a union of $\coker(a)$-classes, it follows that $U'$ is the union of such lower $a$-blocks, and hence $k=|U|$ is even.  
The~$k+1$ largest elements of $Y_1$ are $x<u_1<\cdots<u_k$ (this is just $x$ if $k=0$).  Since $Y_1$ alternates in parity, and since $k$ is even, it follows that $x\equiv \max(Y_1)\bmod2$.
Applying Lemma \ref{lem:opp} to the even-gap set $Y$, noting that $Y_{\leq\max(Y_1)} = Y_1$ gives $\max(Y_1)\equiv|Y_1|\bmod2$, and hence $x \equiv \max(Y_1) \equiv |Y_1|\bmod2$, as required.

\medskip

Case 3: Finally, suppose that $a$ contains a lower block $\{i',x'\}$ where $i<x$; see Figure~\ref{fig:2c}.  By definition of $Z_0$, no element of $Z_0''$ belongs to a floating component of $\Pi(a,b)$, so it follows (by our standing assumption that $i\in Z_0$) that $x\not=i+1$.  Thus, vertex $(i+1)'$ is nested by the $a$-block $\{i',x'\}$ and then by planarity $a$ contains a lower block $\{(i+1)',y'\}$ with $i<i+1<y<x$.  Since $i+1\in Z_0$, and since $X\cup Z_0'$ is a union of $a$-blocks, we also have $y\in Z_0$.  By Lemma \ref{lem:parity},  $y\not\equiv i+1\bmod2$, and hence $y\equiv i\bmod2$. To verify \eqref{eq:iY1} in this case we therefore aim to show that $y \not\equiv |Y_1| \bmod 2$.
This time we define
\[
V = \set{v\in Y_2+2}{i+1<v<y} = \{v_1<\cdots<v_k\} \subseteq Z_0.
\]
Arguing as in the previous case, we see that $k$ is even, and that the $k+1$ smallest elements of $Y_2+2$ are $v_1<\cdots<v_k<y$, so that ${y\equiv \min(Y_2+2) \equiv \min(Y_2) \bmod2}$.  Applying Lemma~\ref{lem:opp} to the even-gap set $Y$, noting that $Y_1 = Y_{<\min(Y_2)}$ gives $\min(Y_2)\not\equiv|Y_1|\bmod2$, and hence ${y \equiv \min(Y_2)\not\equiv|Y_1|\bmod2}$, as required.
\end{proof}

\begin{figure}[ht]
\nc\sss{.43}
\begin{subfigure}[t]{0.45\linewidth}
\begin{center}
\begin{tikzpicture}[scale=\sss]
\uvs{1,3,6,8}
\lvs{1,3,4,6}
\udotted13\udotted68
\ddotted13\ddotted46
\stlines{1/1,3/3,6/4,8/6}
\uarcc45{red}
\node[red] () at (4.1,2.6) {\footnotesize $\phantom{''}i''$};
\draw[|-|] (0,2)--(0,0); \node[left] () at (0,1) {$b$};
\draw[|-] (0,4)--(0,2); \node[left] () at (0,3) {$a$};
\begin{scope}[shift={(0,2)}]
\nc\xx{2.5}
\uvc\xx{red}
\lvc4{red}
\lvc5{red}
\stlinec\xx4{red}
\node[red] () at (\xx,2.5) {\footnotesize $x$};
\end{scope}
\end{tikzpicture}
\end{center}
\caption{Case 1: $a$ has the transversal~$\{x,i'\}$.} \label{fig:2a}
  \end{subfigure}%
  \hspace*{\fill}   
  \begin{subfigure}[t]{0.45\linewidth}
\begin{center}
   \begin{tikzpicture}[scale=\sss]
\uvs{1,3,5,7,10,12}
\uvc4{red}
\uvc8{red}
\uvc9{red}
\lvc4{red}
\lvs{1,3,5,7,8,10}
\udotted13\udotted57\udotted{10}{12}
\ddotted13\ddotted57\ddotted{8}{10}
\stlines{1/1,3/3,5/5,7/7,10/8,12/10}
\stlinec44{red}
\uarcc89{red}
\node[red] () at (3.8,2.6) {\footnotesize $\phantom{''}x''$};
\node[red] () at (8.2,2.6) {\footnotesize $\phantom{''}i''$};
\draw[|-|] (0,2)--(0,0); \node[left] () at (0,1) {$b$};
\draw[|-] (0,4)--(0,2); \node[left] () at (0,3) {$a$};
\begin{scope}[shift={(0,2)}]
\darcxc48{.6}{red}
\end{scope}
\end{tikzpicture}
\end{center}
    \caption{Case 2: $a$ has the block~$\{x'<i'\}$.} \label{fig:2b}
  \end{subfigure}
  \begin{subfigure}[t]{0.7\linewidth}
\begin{center}
\begin{tikzpicture}[scale=\sss]
\uvs{1,3,4,5,6,8,9,10,12,13,14,16}
\lvs{1,3,4,6,7,8,10,11,12,14}
\uvc4{red}
\uvc5{red}
\uvc9{red}
\uvc{13}{red}
\lvc7{red}
\lvc{11}{red}
\udotted13
\udotted68
\udotted{10}{12}
\udotted{14}{16}
\ddotted13
\ddotted46
\ddotted8{10}
\ddotted{12}{14}
\stlines{1/1,3/3,6/4,8/6,10/8,12/10,14/12,16/14}
\stlinec97{red}
\stlinec{13}{11}{red}
\uarcc45{red}
\node[red] () at (3.75,2.7) {\footnotesize $\phantom{''}i''$};
\node[red] () at (9.2,2.6) {\footnotesize $\phantom{''}y''$};
\node[red] () at (13.25,2.65) {\footnotesize $\phantom{''}x''$};
\draw[|-|] (0,2)--(0,0); \node[left] () at (0,1) {$b$};
\draw[|-] (0,4)--(0,2); \node[left] () at (0,3) {$a$};
\node[white] () at (0,5) {$a$}; 
\begin{scope}[shift={(0,2)}]
\darcxc59{.6}{red}
\darcxc4{13}{1.2}{red}
\end{scope}
\end{tikzpicture}  
\end{center}
    \caption{Case 3: $a$ has the blocks~$\{i'<x'\}$ and $\{(i+1)'<y'\}$.} \label{fig:2c}
  \end{subfigure}%
\caption{An illustration of the product graph $\Pi(a,b)$ from the proof of Lemma \ref{lem:fn2}.  The connected component containing $i''$ is shown in red.  Only the relevant edge(s) from $a$ are shown; in Cases 1 and 2, the $a$-block containing $(i+1)'$ could be a transversal or a lower block.}  \label{fig:2}
\end{figure}
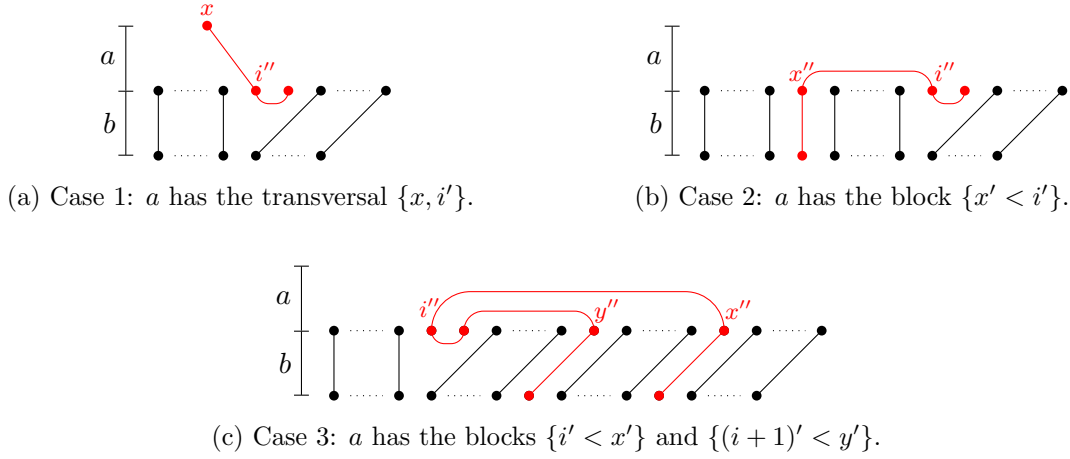

We are now almost ready to prove the main result of this section.  But first we record the following simple consequence of known results.  For $n\in\NN$ and $1\leq i\leq n-1$, we define the Temperley--Lieb partition
\[
e_{i,n} = \custpartn{1,3,4,5,6,8}{1,3,4,6}{\stline11\stline33\stline64\stline86\uarc45\udotted13\ddotted13 \ddotted46\udotted68 \vertlab11 \vertlab4i\vertlab8n} \in \TL_{n,n-2}.
\]
(Of course such an element only exists if $n\geq2$.)

\begin{prop}\label{prop:Om}
The Temperley--Lieb category $\TL$ is generated (as a category) by the set
\[
\Om = \set{e_{i,n},e_{i,n}^*}{n\in\NN,\ 1\leq i\leq n-1}.
\]
\end{prop}

\begin{proof}
Among other things, \cite[Theorem 3.22]{E2024} states that $\TL$ is generated by the elements:
\bit
\item $e_{n-1,n}$ and $e_{n-1,n}^*$ for $n\in\NN$, and
\item $h_{i,n} = \custpartn{1,3,4,5,6,8}{1,3,4,5,6,8}{\stline11\stline33\stline66\stline88\uarc45\darc45\udotted13\ddotted13\udotted68\ddotted68\vertlab11\vertlab4i\vertlab8n}$ for $n\in\NN$ and $1\leq i\leq n-1$.
\eit
The elements $e_{n-1,n}$ and $e_{n-1,n}^*$ of course belong to our stated generating set, and we have $h_{i,n} = e_{i,n}e_{i,n}^*$.
\end{proof}

\begin{thm}
\label{thm:TLfib}
Let $K$ be a non-trivial idempotent semiring.  Then
\[
\mu = \mu_K:\TL\to\M(K):a\mt\ool a
\]
is a faithful involutive category representation. In particular, for each $n \in \NN$ the map $\mu$ gives a faithful involutive representation of the monoid $\TL_n$ by $f_n\times f_n$ matrices over $K$.
\end{thm}

\begin{proof}
It is clear that $\mu$ maps identity partitions to identity matrices, and respects the involutions.  To show that $\mu$ is a morphism we make use of the generating set $\Om$ from Proposition \ref{prop:Om}.  Note that Lemmas \ref{lem:fn1} and \ref{lem:fn2} combine to say that
\begin{equation}\label{eq:bgi}
\ool{ab} = \ool a\, \ool b \qquad\text{for any $a\in\TL$ and $b\in\Om$ with } \br(a)=\bd(b).
\end{equation}
Now consider an arbitrary composable pair $a,b\in\TL$.  By Proposition \ref{prop:Om}, we can write $b = g_1\cdots g_k$ for some $g_1,\ldots,g_k\in\Om$ with each $\br(g_i)=\bd(g_{i+1})$.  Using \eqref{eq:bgi} and induction, we have $\ool b = \ool {g_1}\cdots \ool {g_k}$.  Another induction gives ${\ool{ab} = \ool{ag_{1}\cdots g_{k}} = \ool a \, \ool {g_{1}}\cdots \ool {g_{k}} = \ool a\, \ool b}$.

Now that we have proved that $\mu$ is a morphism, it remains to show injectivity.  To do so, let $a \in \TL_{m,n}$; as usual, we show that we can recover the blocks of $a$ from $\ool a$.  First note that $\dom(a) \in \X_m$ (see Remark \ref{rem:eg}), and that this is the minimal set $X \in \X_m$ such that $\ool{a}_{X,[n]} = 1$.  Thus, we can recover $\dom(a)$ from $\ool a$, and by symmetry we can also recover $\codom(a)$.  One can then recover all transversals (if any) of $a$ by planarity, as in the proof of Theorem \ref{thm:TL_even}.

Next we claim that if we can identify the greatest element in each upper block of $a$, then we can deduce the upper blocks themselves. Indeed, suppose the greatest elements of the upper blocks are $x_1 < x_2 < \dots < x_k$. Since~$x_1$ is the greatest element of some upper block, it cannot be $1$. Consider the vertex ${y_1 = x_1 - 1}$. Clearly $y_1$ cannot be the greatest element of an upper block (by the minimality assumption on $x_1$), and also cannot be part of a transversal edge or the lesser element of a block with an element other than $x_1$ (by planarity, since the connection entailed would cross that of the block containing $x_1$). Hence, the only possibility is that $y_1$ forms an upper block with~$x_1$. Now ignore the vertices $x_1$ and $y_1 (= x_1-1)$ and repeat the argument for each $x_i$ in turn, setting at each stage $y_i$ to be the greatest vertex that is less than $x_i$ and is not equal to $x_j$ or $y_j$ for any $j < i$. We see that each~$x_i$ necessarily pairs up with this $y_i$. This establishes the claim.  A symmetrical argument proves the analogous claim regarding lower blocks.

By symmetry, it remains to show that we can indeed identify the greatest element in each upper $a$-block.  We must now treat separate cases, according to the (equal) parity of $m$ and~$n$.

\medskip

Case 1: $m$ and $n$ are even.  Here we have $\emptyset \in \X_n$. Let $1 \leq i<j \leq m$ where $i$ is odd and $j$ is even. Then $\{i,j\} \in \X_m$, and $\ool{a}_{\{i,j\},\es} = 1$ if and only if $\{i,j\}$ is an upper $a$-block. In this way we can find all upper $a$-blocks whose least element is odd or equivalently all upper $a$-blocks whose greatest element is even. The odd elements of $[m]$ that neither lie in a transversal nor appear as the least element of an upper block must therefore be the greatest element of an upper block. In this way, we identify the greatest element of each upper block.

\medskip

Case 2: $m$ and $n$ are odd.  Here $a$ must contain at least one transversal.  Let $\{x,y'\}$ be the first transversal of $a$; thus $x$ and $y$ are both odd, and we consequently have $\{x\} \in \X_m$ and $\{y\}\in\X_n$.  We show how to identify the greatest element of any upper $a$-blocks lying to the left or right of this transversal.  

First, if $1 \leq i<j < x$ where $i$ is odd and~$j$ is even then $\{i,j, x\} \in \X_m$, and $\ool{a}_{\{i,j,x\}, \{y\}} = 1$ if and only if $\{i,j\}$ is an upper~$a$-block. In this way we can find all upper $a$-blocks lying to the left of $x$ whose least element is odd and whose greatest element is even. Since $\{x,y'\}$ is the first transversal of $a$, the odd elements of $\{1,\ldots, x-1\}$ that do not appear as the least element of an upper $a$-block must therefore be the greatest element of an upper block. Thus, we have identified the greatest element of each upper block lying to the left of $x$. 

Similarly, if $x< i<j \leq n$ where $i$ is even and $j$ is odd, then $\{x,i,j\} \in \X_m$, and $\ool{a}_{\{x,i,j\}, \{y\}} = 1$ if and only if $\{i,j\}$ is an upper~$a$-block. In this way we can find all upper $a$-blocks lying to the right of~$y$ whose least element is even and whose greatest element is odd. The even elements of $\{x+1,\ldots, m\}$ that do not appear in a transversal or as the least element of an upper $a$-block must therefore be the greatest element of some upper block. Thus, we have also identified the greatest element of each upper block lying to the right of $y$. 
\end{proof}

\begin{rem}\label{rem:nonrepTLnPhi}
Unlike the situation with the partition and Brauer categories, the above representation $a\mt\ool a$ of $\TL$ does not lead to analogous representations $(i,a)\mt 2^i\ool a$ of the twisted categories $\TL^\Phi$ or $\TL^{\Phi,d}$ (over rings of appropriate characteristics).  For example, consider the Temperley--Lieb partition $a = \custpartn{1,2,3,4}{1,2,3,4}{\uarcx14{.7}\uarc23\darcx14{.7}\darc23} \in \TL_4$.  Ordering the elements of $\X_4$ as $\emptyset,\{1,4\},\{1,2,3,4\},\{1,2\},\{3,4\}$, we then have
\[
\ool a = 
\begin{pmatrix}
1&1&1&0&0\\
1&1&1&0&0\\
1&1&1&0&0\\
0&0&0&0&0\\
0&0&0&0&0
\end{pmatrix}.
\]
In $\TL_4^\Phi$ we have $(0,a)(0,a)=(2,a)$, but in any semiring of characteristic $0$ we have ${2^0\ool a\cdot 2^0\ool a = 3\ool a \not= 2^{\Phi(a,a)}\ool a}$.
The reason for this `failure' is that although $\ool a$ is indeed a submatrix of $\ol a$ for any $a\in\TL_n$, it does not feature in a block decomposition of $\ol a $ of the kind used in the proof of Corollary \ref{cor:twistBrauer}.

More generally, $\ool a\cdot\ool b$ need not even be a scalar multiple of $\ool{ab}$ (over a ring of characteristic~$0$), for a composable pair $a,b\in\TL$.  For example, with ${b = \custpartn{1,2,3,4}{1,2,3,4}{\stline11\stline44\uarc23\darc23} \in \TL_4}$, and with the same ordering on $\X_4$, we have
\[
\ool{b^2} = \ool b = 
\begin{pmatrix}
1&0&0&0&0\\
0&1&1&0&0\\
0&1&1&0&0\\
0&0&0&0&0\\
0&0&0&0&0
\end{pmatrix}
\AND
\ool b^2 = 
\begin{pmatrix}
1&0&0&0&0\\
0&2&2&0&0\\
0&2&2&0&0\\
0&0&0&0&0\\
0&0&0&0&0
\end{pmatrix}.
\]
\end{rem}

\section{Representations of diagram algebras}\label{sect:algebras}

In this section we relate our main results to known results on the (classical) representation theory of diagram algebras.  The latter are endomorphism algebras in natural linear versions of our diagram categories, whose definitions we now recall.  These can be seen as special instances of a categorical version of the \emph{twisted semigroup algebras} considered for example in \cite{W2007}.

\begin{defn}
Let $K$ be a commutative unital ring, and $\de\in K$ a fixed element.  The \emph{linear partition category (with respect to $K$ and $\de$)}, denoted~$\P(K,\de)$, has:
\bit
\item object set $\NN$,
\item morphism sets $\P_{m,n}(K,\de)$ consisting of all formal $K$-linear combinations of partitions from $\P_{m,n}$,
\item composition $\pr$ defined to be the $K$-linear extension of the rule:
\[
a\pr b = \de^{\Phi(a,b)}ab \qquad{\text{for composable partitions $a,b \in \P$.}}
\]
Here $ab$ is the product in $\P$, and $\Phi(a,b)\in\NN$ is the twisting parameter from Definition \ref{defn:Phi}.
\eit
In fact, $\P(K,\de)$ is an involutive tensor category, where the involution and tensor operation are the $K$-linear extensions of those on $\P$.

We also have the corresponding linear Brauer and Temperley--Lieb categories, $\B(K,\de)$ and $\TL(K,\de)$, defined as appropriate subcategories.

Endomorphism algebras in these linear categories are the partition, Brauer and Temperley--Lieb algebras, $\P_n(K,\de)$, $\B_n(K,\de)$ and $\TL_n(K,\de)$ for each $n\in\NN$.
\end{defn}

For the rest of this section, $K$ is a commutative unital ring of characteristic $0$.
Taking $\de=2\in K$, and arguing as in Theorem \ref{thm:twistpartition} it is easy to see that the mapping $a \mapsto \overline{a}$ induces a $K$-linear involutive tensor representation of $\P(K, 2)$ in the matrix category~$\M(K)$. However, this representation is not faithful as the following example demonstrates.

\begin{ex}\label{ex:P2}
Consider the partition algebra $\P_2(K,2)$, and consider again the partitions $a,b,c\in\P_2$ from Example \ref{ex:P2_0}. Using \eqref{eq:P2}, we have
\[
\ol a + \ol b + \ol c = \begin{pmatrix}
\fourrow3111\\
\fourrow1001\\
\fourrow1001\\
\fourrow1113
\end{pmatrix}.
\]
Since this matrix is symmetric, it follows that
\[
\ol a + \ol b + \ol c = (\ol a + \ol b + \ol c)^T = \ol{a^*} + \ol{b^*} + \ol{c^*},
\]
even though $a+b+c \not= a^* + b^* + c^*$ in $\P_2(K,2)$.  One can generalise this to algebras $\P_n(K,2)$ with $n\geq2$ by fixing a proper subset $\es \subset X\subset[n]$, and taking the partitions
\[
a = \begin{partn}{2} [n] & \\ \hhline{~|-} X & X^c \end{partn} \COMMA
b = \begin{partn}{2} [n] & \\ \hhline{~|-} X^c & X \end{partn} \AND
c = \begin{partn}{2} X & X^c \\ \hhline{-|-} \multicolumn{2}{c}{[n]} \end{partn} .
\]
\end{ex}

The representation $\P(K,2)\to\M(K)$ induced by the mapping $a \mapsto \bar{a}$ restricts to give representations of the Brauer and Temperley--Lieb categories, $\B(K, 2)$ and $\TL(K, 2)$, and it is natural to ask whether these representations are faithful.  We now show that the representation $\B(K,2)\to\M(K)$ is not faithful.  The partitions we use were found using the {\sc Semigroups} package for GAP \cite{Semigroups,GAP}.

\begin{ex}\label{ex:B3}
Consider the following partitions, all from $\B_3$:
\[
a_1 = \bpt311223 \COMMa
a_2 = \bpt221313 \COMMa
a_3 = \bpp132132 \COMMa
b_1 = \bpt321213 \COMMa
b_2 = \bpt211323 \ANd
b_3 = \bpp132231 .
\]
Indexing the subsets of $[3]$ in the order $\es$, $\{1\}$, $\{2\}$, $\{3\}$, $\{1,2\}$, $\{1,3\}$, $\{2,3\}$, $\{1,2,3\}$, one can check (using GAP for example) that
\[
\ol{a_1} + \ol{a_2} + \ol{a_3} = \begin{pmatrix}
3 & 0 & 0 & 0 & 0 & 1 & 1 & 0 \\
0 & 0 & 0 & 1 & 0 & 0 & 0 & 0 \\
0 & 1 & 1 & 0 & 0 & 0 & 0 & 1 \\
0 & 1 & 1 & 0 & 0 & 0 & 0 & 1 \\
1 & 0 & 0 & 0 & 0 & 1 & 1 & 0 \\
1 & 0 & 0 & 0 & 0 & 1 & 1 & 0 \\
0 & 0 & 0 & 0 & 1 & 0 & 0 & 0 \\
0 & 1 & 1 & 0 & 0 & 0 & 0 & 3 
\end{pmatrix}
= \ol{b_1} + \ol{b_2} + \ol{b_3},
\]
even though $a_1+a_2+a_3 \not= b_1+b_2+b_3$ in $\B_3(K,2)$.  The above partitions might appear ad hoc, but one can check that we have $b_i = a_is$ for $i=1,2,3$, for the permutation~${s = \bpp122133}$.  Right multiplication by the permutation matrix $\ol s$ swaps columns $\{1\}\leftrightarrow\{2\}$ and $\{1,3\}\leftrightarrow\{2,3\}$, and it is apparent that the above matrix $\ol{a_1} + \ol{a_2} + \ol{a_3}$ is invariant under these swaps.
\end{ex}

Unlike the situation with the linear categories $\P(K,2)$ and $\B(K,2)$, it turns out that the representation
\begin{equation}\label{eq:tau}
\TL(K,2)\to\M(K):a\mt\ol a
\end{equation}
of the linear Temperley--Lieb category \emph{is} faithful, as we now explain.

In the case where $\delta$ is a natural number greater than or equal to $2$, a faithful representation of the Temperley--Lieb \emph{algebra} $\TL_n(K,\delta)$ mapping each basis element ${a \in \TL_n}$ to a zero-one matrix of size $\delta^n\times\delta^n$ is known.  This representation was first introduced by Brauer \cite{B1937}, and shown to be faithful in \cite{J1994} and \cite{DKP2006}.  When $\de=2$, these representations turn out to be restrictions of our representation~\eqref{eq:tau}, as we now show, following the presentation given in \cite{DKP2006}.  We then give a short argument to show that the full representation \eqref{eq:tau} is faithful by making use of the fact that the restriction to each Temperley--Lieb algebra is known to be faithful.

The Temperley--Lieb algebra $\TL_n(K, 2)$ is spanned by the elements of $\TL_n$, which in turn are generated by the elements
\[
h_{i,n} = \custpartn{1,3,4,5,6,8}{1,3,4,5,6,8}{\stline11\stline33\stline66\stline88\uarc45\darc45\udotted13\ddotted13\udotted68\ddotted68\vertlab11\vertlab4i\vertlab8n} \qquad\text{for $1\leq i\leq n-1$.}
\]
(We saw these elements in the proof of Proposition \ref{prop:Om}.)  For $1\leq i\leq n-1$, Brauer's representation of $\TL_n(K,2)$ maps the generators according to the rule
\begin{equation}\label{eq:Hin}
h_{i,n} \mt H_{i,n} = I_{2^{i-1}}\otimes M \otimes I_{2^{n-i-1}} \in \M_{2^n}(K), \qquad\text{where}\qquad M = \begin{pmatrix} 1&0&0&1\\0&0&0&0\\0&0&0&0\\1&0&0&1\end{pmatrix}.
\end{equation}
(Actually, Brauer maps $h_{i,n}$ to $I_{2^{n-i-1}}\otimes M \otimes I_{2^{i-1}}$; note the swapping of the identity matrices.  However, each of the two representations can be obtained from the other by first applying the isomorphism $\TL_n\to\TL_n$ given by $h_{i,n}\mt h_{n-i,n}$.)  

To see that the representation \eqref{eq:Hin} is a restriction of the representation \eqref{eq:tau}, it suffices to show that the two representations agree on the generators $h_{i,n}$.  But this follows quickly from the fact that \eqref{eq:tau} is a tensor representation, provided we order the subsets of each $[n]$ carefully.  This ordering is given inductively:
\bit
\item There is only one subset of $[0] = \es$, namely $\es$ itself.
\item Now suppose we have ordered the subsets of some $[n]$ as $A_1<\cdots<A_{2^n}$.  The ordering on $[n+1]$ is then given by
\[
A_1<\cdots<A_{2^n}<A_1\cup\{n+1\}<\cdots<A_{2^n}\cup\{n+1\}.
\]
\eit
And now, writing ${h = h_{1,2} = \custpartn{1,2}{1,2}{\uarc12\darc12} \in \TL_2}$, one can easily check that $\ol h = M$, and so indeed
\[
\ol{h_{i,n}} = \ol{\io_{i-1} \op h \op \io_{n-i-1}} = \ol{\io_{i-1}} \ot \ol h \ot \ol{\io_{n-i-1}} = I_{2^{i-1}} \ot M \ot I_{2^{n-i-1}} = H_{i,n}.
\]
Using this, we can now prove the following.

\begin{thm}\label{thm:TLC}
If $K$ is a field of characteristic $0$, then $a\mt\ol a$ determines a faithful involutive $K$-linear tensor category representation $\TL(K,2)\to\M(K)$.
\end{thm}

\begin{proof}
It remains only to establish injectivity, so suppose $a,b\in\TL_{m,n}(K,2)$ are such that $\ol a = \ol b$.  (Here $a$ and $b$ are $K$-linear combinations of basis elements from $\TL_{m,n}$, and we write $\ol a$ and $\ol b$ for the corresponding linear combinations of matrices in $\M(K)$.)  By symmetry we may assume that $m\geq n$.  Define the Temperley--Lieb element
\[
c = \custpartn{1,3,4,5,7,8}{1,3}{\stline11\stline33\uarc45\uarc78\udotted13\ddotted13\udotted57\vertlab11\vertlab3n\vertlab8m} \in \TL_{m,n},
\]
and note that $c^*\pr c = 2^k \io_n$, where $k = (m-n)/2$ is the number of upper $c$-blocks.  From $\ol a=\ol b$, and $\Phi(d,c^*) = 0$ for each basis element $d \in \TL_{m,n}$, we deduce
\[
\ol{a\pr c^*} = \ol{ac^*} = \ol a\, \ol{c^*} = \ol b\, \ol{c^*} = \ol{bc^*} = \ol{b\pr c^*}.
\]
But $a\pr c^*$ and $b\pr c^*$ are elements of the Temperley--Lieb \emph{algebra} $\TL_m(K,2)$, so it follows by the faithfulness of Brauer's representation that $a\pr c^* = b\pr c^*$.  This in turn gives
\[
2^ka = a\pr(2^k\io_n) = a\pr c^*\pr c = b\pr c^*\pr c = b\pr(2^k\io_n) = 2^kb,
\]
and so $a=b$, as required.
\end{proof}

\section{Concluding remarks and open questions}\label{sect:conclusion}

Our main results give faithful linear representations of the diagram categories~$\P$,~$\B$ and~$\TL$, as well as their ($d$-)twisted counterparts, over various kinds of semiring.  Where these representations preserve the involution and tensor operations, we have seen that the
corresponding sequence of dimensions of matrices grows exponentially and is minimum possible for such a representation. However, it remains open whether smaller representations are possible if one does not insist on preserving these operations.

\begin{ques}
Does there exist a faithful matrix representation of $\P$ with sequence of dimensions less than $2^n$, or even growing sub-exponentially?
\end{ques}

Note that our reduced (non-tensor-preserving) representations of $\B$ and $\TL$ from Sections \ref{sect:Bn} and \ref{sect:TLn} are still exponential.  (Recall that the Fibonacci number $f_n$ is asymptotic to $\frac1{\sqrt5}\left(\frac{1+\sqrt5}2\right)^{n+1}$.) Again, we can ask if this is minimal.

\begin{ques}
Does there exist a faithful matrix representation of $\TL$ with sequence of dimensions less than the Fibonacci bound given by our results, or even growing sub-exponentially?
\end{ques}

One can also ask whether it is possible to move beyond idempotent semirings.  Given Remark \ref{rem:conv}, a positive answer to the next question would involve building a completely different representation.

\begin{ques}
Do there exist `natural' faithful (involutive tensor category) representations of $\P$, $\B$ or $\TL$ in $\M(K)$ where $K$ is a \emph{ring}, or even a \emph{field}?
\end{ques}

Analogous questions can also be asked for the corresponding twisted categories.

Of course minimum-dimension representations for individual diagram \emph{monoids} will not necessarily arise as restrictions of representations of the corresponding  categories, so there also arise separate questions about their effective dimension \eqref{eq:dim}.  (As we noted in the introduction, any finite monoid---such as $\P_n$, $\B_n$ or $\TL_n$---embeds in some transformation monoid $\T_d$, and hence in the matrix monoid $\M_d(K)$ over \emph{any} semiring~$K$.)

\begin{ques}
Given a fixed semiring $K$, what are the effective dimensions of the diagram monoids $\P_n$,~$\B_n$ and~$\TL_n$ (and their ($d$-)twisted counterparts) over $K$?  
\end{ques}

For example, given an idempotent semiring $K$, Theorems~\ref{thm:partition},~\ref{thm:Brauer_odd} and~\ref{thm:TLfib} give  
\[
{\rm dim}_K(\P_n) \leq 2^n \COMMA {\rm dim}_K(\B_n) \leq 2^{2\lfloor n/2\rfloor} \AND  {\rm dim}_K(\TL_{n}) \leq f_n.
\]
One could ask if we have equality here, at least for suitably large $n$.\footnote{For trivially small values of $n\leq1$ the above bounds give $\dim_K(\P_0)\leq1$ and $\dim_K(\P_1)\leq2$.  It is easy to see, however, that we have isomorphisms of partition and relation monoids, $\P_0\cong\R_0$ and $\P_1\cong\R_1$, so that in fact $\dim_K(\P_0)=0$ and $\dim_K(\P_1)=1$ for any idempotent semiring $K$.  Similarly, it is easy to embed the Brauer monoid $\B_2 = \left\{\smallcustpartn{1,2}{1,2}{\stline11\stline22},\smallcustpartn{1,2}{1,2}{\stline12\stline21},\smallcustpartn{1,2}{1,2}{\uarc12\darc12}\right\}$ in $\R_2$ (but not in $\R_1$), so that $\deg_K(\B_2)=2(<2^2)$.}  It is possible to investigate this question for (very) small values of $n$ computationally, using the {\sc Semigroups} package for GAP \cite{GAP,Semigroups}.  As we mentioned in Section \ref{sect:intro}, such computations reveal that when $n=2$ the partition monoid $\P_2$ cannot be faithfully represented in the relation monoid $\R_3$, and hence we have the exact value $\deg_\BB(\P_2) = 4 (=2^2)$, where $\BB$ is the \emph{Boolean} semiring.

One could also study the effective dimension if $K$ is allowed to vary across some appropriate collection of semirings. It is easy to see that every semigroup embeds in the multiplicative semigroup of \textit{some} semiring\footnote{For example, the monoid algebra over any field of the monoid obtained by adjoining an identity element to the semigroup. The identity element here is necessary to ensure that the algebra has a multiplicative identity element, and hence meets our definition of a semiring.}, and so if the semiring is allowed to be completely arbitrary then every semigroup has effective dimension $1$.  However, one might impose natural restrictions on the ground (semi)ring, and for example ask:

\begin{ques}
What are the minimum effective dimensions of the diagram monoids $\P_n$,~$\B_n$ and~$\TL_n$ (and their ($d$-)twisted counterparts), considered across all \emph{commutative semirings}?  Or across all \emph{fields}?
\end{ques}

The paper \cite{MS2012a} contains general results on effective dimensions of finite semigroups over fields.  Among many other things, the computation of the effective dimension over an algebraically closed field $K$ is reduced to the (decidable) non/equational theory of~$K$, although it is noted that `it is not feasible in practice to compute the effective dimension of your favourite semigroup' in this way.  Effective dimensions are calculated in \cite{MS2012a} for several families of monoids (such as full/partial transformation monoids, symmetric inverse monoids and matrix monoids over finite fields), but these tend to belong to the class of \emph{Rhodes semisimple} monoids, which (for reasons explained in \cite{MS2023,CEM2024}) does not contain our diagram monoids.

\section*{Acknowledgements}

The authors thank the Dame Kathleen Ollerenshaw Trust for financial support; particularly funding a research visit of the first author to the University of Manchester where this research was carried out, and for supporting conference travel of the second author.  The research of the first author was partially supported by the Australian Research Council [Future Fellowship FT190100632].  The research of the third author was supported by the Engineering and Physical Sciences Research Council [grant number EP/Y008626/1].  We also thank J.S.~Lemay for some helpful conversations.

\end{document}